\documentclass[a4paper,10pt,oneside  ]{article}	
\makeindex

\usepackage[utf8]{inputenc}
\usepackage{amsthm}
 \usepackage{a4wide}
\usepackage{amsmath}
\usepackage{amssymb}
\usepackage{graphicx}
\usepackage{xcolor}
\usepackage{epsfig}

\theoremstyle{plain}

\newtheorem{theorem}{Theorem}[section]

\newtheorem{lemma}[theorem]{Lemma}
\newtheorem{proposition}[theorem]{Proposition}

\theoremstyle{definition}

\newtheorem{remark}[theorem]{Remark}

\numberwithin{equation}{section}
\numberwithin{figure}{section}

\newcommand{\eps}{\varepsilon}
\newcommand{\e}{\varepsilon}

\newcommand{\R}{\mathbb {R}}
\newcommand{\N}{\mathbb N}
\newcommand{\A}{\mathcal A}
\newcommand{\Z}{\mathbb Z}
\newcommand{\wto}{\rightharpoonup}
\newcommand{\dom}{\operatorname{dom}}
\newcommand{\argmin}{\operatorname{argmin}}
\newcommand{\loc}{\operatorname{loc}}

 \newcommand{\ccb}{\color{black}}

\title{On Lennard-Jones systems with finite range interactions and their asymptotic analysis}
\author{M. Schäffner\footnote{TU Dresden, Department of Mathematics, Zellescher Weg 12-14, 01069 Dresden, Germany mathias.schaeffner@tu-dresden.de}, A. Schlömerkemper\footnote{University of W\"urzburg, Institute of Mathematics, Emil-Fischer-Stra\ss{}e 40, 97074 Würzburg, Germany  anja.schloemerkemper@mathematik.uni-wuerzburg.de}}

\begin{document}

\maketitle
\begin{abstract}
The aim of this work is to provide further insight into the qualitative behavior of mechanical systems that are well described by Lennard-Jones type interactions on an atomistic scale. By means of $\Gamma$-convergence techniques, we study the continuum limit of one-dimensional chains of atoms with finite range interactions of Lennard-Jones type, including the classical Lennard-Jones potentials. So far, explicit formulae for the continuum limit were only available for the case of nearest and next-to-nearest neighbour interactions. In this work, we provide an explicit expression for the continuum limit in the case of finite range interactions. The obtained homogenization formula is given by the convexification of a Cauchy-Born energy density. 

Furthermore, we study 
rescaled energies in which bulk and surface contributions scale in the same way. The related discrete-to-continuum limit yields a rigorous derivation of a one-dimensional version of Griffith' fracture energy and thus generalizes earlier derivations for nearest and next-to-nearest neighbors to the case of finite range interactions.

A crucial ingredient to our proofs is a novel decomposition of the energy that allows for refined estimates.

\end{abstract}

\section{Introduction}

Our article follows the general aim of deriving continuum theories for mechanical systems from underlying discrete systems, see e.g.\ \cite{Ball}. Here, we are interested in discrete systems with non-convex interaction potentials that allow for fracture of mechanical systems.  One of the first contributions in this direction is due to Truskinovsky \cite{T96}. In that article a chain of atoms which interact by Lennard-Jones potentials is considered and a model for fracture is derived. Later this approach was extended by using the notion of $\Gamma$-convergence in \cite{BDG99,BG02,BG04}. In order to capture surface effects, a refined analysis was performed based on calculating the first order $\Gamma$-limit, see \cite{BC07,SSZ11}, or on studying suitably rescaled energies \cite{BG15,BLO,BS14,BT,FS14,FS15,SSZ12}. 



The main scope of the present paper is to provide a rather explicit description of limiting functionals for discrete systems with Lennard-Jones type interactions of finite range. To make this more precise, we fix some notation. We consider a chain of $n{+}1$ lattice points with $n\in\N$. The interaction of lattice points with distance $j\frac{1}{n}$ in the reference lattice is described by a potential $J_j$, $j\in\{1,\ldots,K\}$ with $K\in \N$. The mathematical assumptions on the potentials $J_j$, $j=1,\ldots,K$, are phrased in Section~\ref{Sec:setting}. In particular these are satisfied if $J_j(z)=J(jz)$ for some Lennard-Jones potential $J(z) = \frac{k_1}{z^{12}} - \frac{k_2}{z^6}$, $k_1,k_2>0$ if $z>0$, and $+\infty$ otherwise. Therefore, we call the potentials which satisfy our assumptions potentials of Lennard-Jones type. The free energy of the system under consideration is the sum of all pair interactions up to range $K$ with the canonical bulk scaling. It reads 
$$H_n(u)=\sum_{j=1}^K\sum_{i=0}^{n-1}\lambda_nJ_j\left(\frac{u^{i+j}-u^i}{j\lambda_n}\right),$$
where $\lambda_n :=\frac1n$ and $u^i$ denotes the deformation of the $i$th lattice point satisfying certain periodic boundary conditions on $[0,1)\cap \lambda_n \mathbb{Z}$ with $u$ being its piecewise affine interpolation.

We are interested in the asymptotic behavior of the system as $n\to \infty$ and therefor consider the $\Gamma$-limit of the sequence of functionals $(H_n)$, see Section~\ref{sec:Hn}. The $\Gamma$-limit of discrete functionals of the form of $H_n$ was derived under very general assumptions on the interaction potentials in \cite{BG04}. The $\Gamma$-limit result of \cite[Theorem 3.2]{BG04} phrased for Lennard-Jones type potentials asserts that $(H_n)$ $\Gamma$-converges to an integral functional $H$, which is defined on the space of functions of bounded variation and has the form
\begin{align} \label{GammalimitH}
H(u)=\int_0^1 \phi(u'(x))dx,
\end{align}
where $\phi$ is defined via some homogenization process which involves minimization of larger and larger 'cells', see Remark~\ref{rem:stella} below for details. In the special case of Lennard-Jones potentials, $\phi$ reads
\begin{align} \label{formula-phi}
 \phi(z):=\lim_{N\to\infty}\min\bigg\{&\frac{1}{N}\sum_{j=1}^K\sum_{i=0}^{N-j}J\left( u^{i+j}-u^i\right)\,|\notag\\
 &\qquad u:\N_0\to\R, u^i=zi\mbox{ if }i\notin\{K+1,\dots, N-K-1\}\bigg\},
\end{align}
cf.~\cite[Theorem~23]{BG07}. 

If $K=1$, $\phi=J_1^{**}$, where $J^{**}$ denotes the lower semicontinuous and convex envelope of $J$, see e.g.\ \cite{BG02}. If $K=2$ it was shown that $\phi = (J_1 + J_2)^{**}$ for Lennard-Jones systems, see e.g.~\cite{BG04}. In this article we extend this result to $K>2$ and prove
$$\phi=J_{CB}^{**},$$ 
with the Cauchy-Born energy density
$$J_{CB}=\sum_{j=1}^KJ_j.$$ 
Thus, the formula \eqref{formula-phi} for $\phi$ has a rather explicit expression for the large class of Lennard-Jones type potentials. Our result also extends a corresponding formula in the case of convex interaction potentials, see \cite{BGS,GPPS}, to a class of nonconvex interaction potentials.   

Let us recall the following major difference between $K>2$ and the case of nearest and next-to-nearest neighbor interactions, i.e.~$K=2$: In the latter case and for rather general interaction potentials $J_j$, the limiting energy density $\phi$ is given by the 'single cell formula' $\phi\equiv J_0^{**}$,  where $J_0$ is an effective potential given by the following infimal convolution-type formula, which takes possible oscillations on the lattice-level into account 
$$J_0(z):=J_2(z)+\tfrac12\inf\{J_1(z_1)+J_1(z_2),z_1+z_2=2z\},$$
see e.g.~\cite[Remark 3.3]{BG04}. With this formula at hand it is not difficult to show that in the case of Lennard-Jones type potentials there are essentially no oscillations on the lattice-level and it holds $\phi(z)=(J_1+J_2)^{**}(z)=J_{CB}^{**}(z)$. However, up to our knowledge, there has not been a result in the literature yet which asserts whether or how the formula for the effective potential $J_0$ extends to a larger interaction range, i.e., to $K>2$.

The key idea in our proof for general finite range interactions is to bypass the absence of a 'single cell formula' for $\phi$ by carefully decomposing the energy into sub-systems, which are then considered separately. For each of such sub-systems an effective potential similar to $J_0$, given above, is available. Appealing to the special convex/concave shape of the Lennard-Jones type potentials it is then possible to exclude oscillations on the lattice-level which justifies a posteriori the energy splitting and enables us to show $\phi=J_{CB}^{**}$ for Lennard-Jones type potentials. In Section~\ref{Sec:setting}, we describe this novel energy decomposition in detail and show $\phi=J_{CB}^{**}$ for Lennard-Jones type potentials in Section~\ref{sec:Hn}.

As an aside, we mention that the pointwise limit of $(H_n)$ in the spirit of \cite{BLBLi} yields a similar integral functional as above with $J_{CB}$ instead of $J_{CB}^{**}$. Hence, roughly speaking, the $\Gamma$--limit is the lower semicontinuous envelope of the pointwise limit in this case.

The above  $\Gamma$--limit result is of interest also in the analysis of computational techniques as the so called quasicontinuum method, see e.g.~\cite{LO,TOP}. The main idea of that method is to couple fully atomistic and continuum descriptions of solids. Many formulations of those models rely on the assumption that the effective energy of the continuum limit of discrete energies like $H_n$ is given by $J_{CB}$. The result of Theorem~\ref{zero2} makes it possible to extend the $\Gamma$-convergence analysis of certain quasicontinuum models for the case $K=2$, see \cite{SS14}, also to the case of general finite range interactions of Lennard-Jones type, see \cite{S15}.

\smallskip

The limiting functional $H$ in \eqref{GammalimitH} captures the bulk contributions of the energy. In order to capture also surface contributions due to the formation of cracks,  it is a well-known strategy to consider suitably rescaled functionals, cf., e.g., \cite{BLO,BS14,SSZ12}. In Section~\ref{sec:En} we thus consider rescaled functionals for which the contribution of elastic deformations and surface contributions due to jumps are on the same order of magnitude. We observe that a global minimizer of $H_n$ is given by the linear function $x\mapsto \gamma x$ with $\{\gamma\}=\argmin J_{CB}$, see \eqref{lb:hnu}.  We define a rescaled functional by
\begin{equation*}
 E_n(v):=\sum_{j=1}^K\sum_{i=0}^{n-1}J_j\left(\gamma+\frac{v^{i+j}-v^i}{j\sqrt{\lambda_n}}\right)-nJ_{CB}(\gamma)
\end{equation*}
with certain periodic boundary conditions, see below, where $v^i$ is a scaled version of the displacement of lattice point $i$ from its equilibrium configuration $\gamma i$, and $v$ is its piecewise affine interpolation. The above energy $E_n$ is a variation of the energy considered in \cite[Theorem~4]{BLO} for multibody potentials with finite range interaction, and in \cite[Theorem~6.1]{SSZ12} for nearest and next to nearest neighbor interactions and Dirichlet type boundary conditions. \ccb The result in \cite{BLO} seems not directly applicable to pair potentials as e.g.~the classical Lennard-Jones potentials if $K>2$, see Remark~\ref{blofail} below or \cite[Remark 3]{BLO} and \cite[Section 4]{BS14}. Combining the decomposition of the energy mentioned above with the line of arguments of \cite[Theorem~4]{BLO}, we prove that $(E_n)$ $\Gamma$-converges to a one-dimensional version $E$ of the Griffith energy for fracture in the case of Lennard-Jones type potentials and $K>2$:
\begin{equation*}
E(v)=\tfrac12J_{CB}''(\gamma)\int_0^1 v'(x)^2\,dx+\beta \# S_v, \quad [v](x)>0\mbox{ for $x\in S_v$,}
\end{equation*}
where $S_v$ is the jump set of $v$ and $\beta$ is some boundary layer energy given in \eqref{def:bij}, cf.\ Theorem~\ref{Th1} for details.

\section{Setting of the problem}\label{Sec:setting}

We consider a one-dimensional lattice given by $\lambda_n\mathbb Z$ with $\lambda_n=\frac{1}{n}$. We denote by $u:\lambda_n\mathbb Z\to \mathbb R$ the deformation of the atoms from the reference configuration and write $u(i\lambda_n)=u^i$ as shorthand. In the following, we identify a discrete deformation $u$ with its piecewise affine interpolation and consider for simplicity only deformations with a $1$-periodic derivative, i.e.~$u\in\A_n^{\#}(0,1)$, where  
\begin{align*}
 \A_n^\#(0,1):=\big\{u\in W_{\loc}^{1,\infty}(\R):u\mbox{ is affine on }(i,i+1)\lambda_n\mbox{ for all $i\in\Z$ and $u'$ is $1$-periodic}\big\}.
\end{align*}
For given $K\in\N$, we define a discrete energy of a deformation $u\in\A_n^\#(0,1)$ by
\begin{equation}\label{ene:bulk}
 H_n(u):=\sum_{j=1}^K\sum_{i=0}^{n-1}\lambda_nJ_j\left(\frac{u^{i+j}-u^i}{j\lambda_n}\right),
\end{equation}
where $J_j$, $j=1,\dots,K$ are potentials of Lennard-Jones type which satisfy
\begin{itemize}
 \item[(i)] $J_j$ is $C^2$ on its domain and $(0,+\infty)\subset\dom J_j$,
 \item[(ii)] $\lim\limits_{z\to-\infty}|z|^{-1}J_j(z)=+\infty$,
 \item[(iii)] $\lim\limits_{z\to+\infty}J_j(z)=0$,
 \item[(iv)] $\argmin\limits_z J_j(z) =\{\delta_j\}$, $J_j(\delta_j)<0$.  
\end{itemize}
Let $u\in\A_n^\#(0,1)$ and $j\in\{2,\dots,K\}$ be given. Appealing to the equality $u^{n+s}-u^{n+s-1}=u^s-u^{s-1}$ (a consequence of $1$-periodicity of $u'$), we can rewrite the nearest neighbour interactions in \eqref{ene:bulk} as follows:
\begin{align}\label{rewJ1}
 \sum_{i=0}^{n-1}J_1\left(\frac{u^{i+1}-u^i}{\lambda_n}\right)
=&\sum_{i=0}^{n-1}\frac{1}{j}\sum_{s=i}^{i+j-1}J_1\left(\frac{u^{s+1}-u^s}{\lambda_n}\right).
\end{align}
Let $c=(c_j)_{j=2}^K$ be such that $\sum_{j=2}^Kc_j=1$. Using \eqref{rewJ1}, we can rewrite the energy \eqref{ene:bulk} as 
\begin{align}\label{ene:rev0}
  H_n(u)
=&\sum_{j=2}^K\sum_{i=0}^{n-1}\lambda_n\left\{J_j\left(\frac{u^{i+j}-u^i}{j\lambda_n}\right)+\frac{c_j}{j} \sum_{s=i}^{i+j-1}J_1\left(\frac{u^{s+1}-u^s}{\lambda_n}\right)\right\}.
\end{align}
For given $j\in\{2,\dots,K\}$, we define the following functions
\begin{equation}\label{J0j}
 J_{0,j}(z):=J_{j}(z)+\frac{c_j}{j}\inf\left\{\sum_{s=1}^jJ_1(z_s),~\sum_{s=1}^jz_s=jz\right\}.
\end{equation}
Note that the definition of $J_{0,j}$ yields lower bounds for the terms in the curved brackets in \eqref{ene:rev0}. Let us remark that in the case of nearest and next-to-nearest neighbour interactions, i.e.~$K=2$, we have $c_2=c_K=1$ and 
$$J_{0,2}(z)=J_2(z)+\frac12\inf\{J_1(z_1)+J_1(z_2):~z_1+z_2=2z\},$$
which is exactly the effective energy density which shows up in \cite{BC07,SSZ11,SSZ12}, and similarly in \cite{BS14}.\\
 Next, we formulate further assumptions on the potentials $J_j$ in terms of the functions $J_{0,j}$:  
\begin{enumerate}
\item[(v)] There exists $c=(c_j)_{j=2}^K\in\R^{K-1}_+$ such that $\sum_{j=2}^Kc_j=1$, and $J_{0,j}$ defined in \eqref{J0j} satisfies the hypotheses (vi), (vii), and (viii) for $j\in\{2,\dots,K\}$. 
\item[(vi)] There exists a unique $\gamma>0$, independent of $j$, such that 
\begin{equation}\label{gamma}
 \{\gamma\}=\argmin_{z\in\R} J_{0,j}(z).
\end{equation}
Furthermore, $J_{0,j}''(\gamma)>0$ and there exists $\eps>0$, independent of $j$, such that
\begin{equation}\label{ass:unique1}
\{(z,\dots,z)\}=\argmin\left\{\sum_{s=1}^jJ_1(z_s),~\sum_{s=1}^jz_s=jz\right\}\qquad\mbox{for all $z\leq \gamma+\eps$}.
\end{equation}
\item[(vii)] There exists $\eta>0$ and $C>0$ such that
\begin{equation}\label{cb}
 J_j(z)+\frac{c_j}{j}\sum_{s=1}^jJ_1(z_s)\geq J_j(z)+c_jJ_1(z)+C\sum_{s=1}^j(z_s-z)^2
\end{equation}
whenever $\sum_{s=1}^jz_s=jz$ and $\sum_{s=1}^j|z_s-z|+|z-\gamma|\leq\eta$.
\item[(viii)] $\liminf\limits_{z\to+\infty}J_{0,j}(z)>J_{0,j}(\gamma)$.
\end{enumerate}

\begin{remark}
Note that a direct consequence of hypothesis (vi) is
\begin{equation}
 J_{0,j}(z)=J_j(z)+c_jJ_1(z)=:\psi_j(z)\quad\mbox{for all $z\leq \gamma+\eps$ and}\quad \psi_j''(\gamma)>0\label{def:psij}
\end{equation}
for all $j\in\{2,\dots,K\}$.
\end{remark}

Assumptions (v)--(vii) are tailor-made in order to rule out certain microscopic relaxation effects which in general might occur for discrete systems with non-convex interaction potentials, see Remark~\ref{rem:stella}. We will show in Proposition \ref{prop:lj} that the classical Lennard-Jones potentials indeed satisfy assumptions (i)--(viii). In Section~\ref{sec:Hn}, we provide a rather explicit expression of the $\Gamma$-limit of $H_n$ subject to additional periodic boundary conditions.\footnote{The last sentence does not seem to make sense here. Can we delete it?}

Note that \eqref{ene:rev0} and the assumptions (v) and (vi) imply
\begin{equation}\label{lb:hnu}
 H_n(u)\geq \sum_{j=2}^K\sum_{i=0}^{n-1}\lambda_n J_{0,j}\left(\frac{u^{i+j}-u^i}{j\lambda_n}\right)\geq \sum_{j=2}^K(J_j(\gamma)+c_jJ_1(\gamma))=\sum_{j=1}^KJ_j(\gamma).
\end{equation}
Hence, $u_{\min}(x)=\gamma x$ is a minimizer of $H_n$ for all $n\in\N$. Let us now consider deformations $u\in\A_n^{\#}(0,1)$ which are close to the equilibrium configuration $u_{\min}$. To this end, set $v^i=\frac{u^i-\lambda_n\gamma i}{\sqrt{\lambda_n}}$ and define
\begin{equation*}
 E_n(v):=\sum_{j=1}^K\sum_{i=0}^{n-1}J_j\left(\gamma+\frac{v^{i+j}-v^i}{j\sqrt{\lambda_n}}\right)-nJ_{CB}(\gamma)=\frac{H_n(u_{\min}+\sqrt{\lambda_n}v)-\inf H_n}{\lambda_n},
\end{equation*}
where $J_{CB}$ is defined by
\begin{equation}\label{def:jcb}
 J_{CB}(z):=\sum_{j=1}^KJ_j(z).
\end{equation}
In Section \ref{sec:En}, we derive a $\Gamma$--limit of $E_n$ as $n$ tends to infinity under additional boundary conditions. We define the sequence of functionals $(E_n^\ell)$ by
\begin{equation}\label{ene:rescaledbc}
 E_n^{\ell}(v):=\begin{cases}
                   E_n(v)&\mbox{if $v\in\A_n^{\#,\ell}(0,1)$,}\\
		   +\infty&\mbox{else,}
                  \end{cases}
\end{equation}
where 
\begin{equation}\label{def:anell}
 \A_n^{\#,\ell}(0,1):=\{v\in \A_n^{\#}(0,1):\mbox{ $x\mapsto v(x)-\ell x$ is $1$-periodic}\}.
\end{equation}
In Theorem \ref{Th1}, we derive the $\Gamma$--limit of the sequence $(E_n^\ell)$ as $n$ tends to infinity.\\

Next we show that the assumptions (i)--(viii) are reasonable in the sense that they are satisfied by the classical Lennard-Jones potentials.
\begin{proposition}\label{prop:lj}
 For $j\in\{1,\dots,K\}$ let $J_j$ be defined as 
\begin{equation}\label{def:lj}
 J_j(z)=J(jz)\quad\mbox{with}\quad J(z)=\frac{k_1}{z^{12}}-\frac{k_2}{z^6},\quad\mbox{for $z>0$ and $J(z)=+\infty$ for $z\leq0$}
\end{equation}
and $k_1,k_2>0$. Then there exists $(c_j)_{j=2}^K$ such that hypotheses (i)--(viii) are satisfied. Moreover, it holds that $\dom J_j=(0,+\infty)$ for $j\in\{1,\dots,K\}$ and that for all $z>0$ and $j\in\{2,\dots,K\}$
\begin{equation}\label{psienv}
 J_{0,j}^{**}(z)=\psi_j^{**}(z)=\begin{cases}
                                  \psi_j(z)&\mbox{if $z\leq\gamma$,}\\
				  \psi_j(\gamma)&\mbox{if $z>\gamma$}.
                                 \end{cases}
\end{equation}
\end{proposition}

\begin{proof}
By the definition of $J_j$, $j=1,\dots,K$ it is clear that they satisfy (i)--(iv) and $\dom J_j=(0,+\infty)$. Note that the unique minimizer $\delta_j$ of $J_j$ is given by
\begin{equation}\label{LJ:deltaj}
 \delta_j=\frac{1}{j}\left(\frac{2k_1}{k_2}\right)^{1/6},
\end{equation}
and $J$ is strictly convex on $(0,z_c)$ with $z_c=(\frac{13}{7})^{\frac16}\delta_1>\delta_1$. Let us show (v)--(viii). The function $J_{CB}$ is given by
$$J_{CB}(z)=\sum_{j=1}^KJ(jz)=\frac{k_1}{z^{12}}\sum_{j=1}^K\frac{1}{j^{12}}-\frac{k_2}{z^6}\sum_{j=1}^K\frac{1}{j^{6}}.$$
Hence, $J_{CB}$ is also a Lennard-Jones potential with the unique minimizer 
\begin{equation}\label{gammaK}
 \gamma=\left(\frac{2k_1}{k_2}\right)^{1/6}\left(\frac{\sum_{j=1}^K\frac{1}{j^{12}}}{\sum_{j=1}^K\frac{1}{j^{6}}}\right)^{1/6}<\delta_1.
\end{equation}
It can be checked that $J'(\gamma)<0$ and $J'(j\gamma)>0$ for every $j\geq2$. We define $(c_j)_{j=2}^K$ as
\begin{equation}\label{cjlj}
 c_j:=-\frac{jJ'(j\gamma)}{J'(\gamma)}>0.
\end{equation}
Since $\gamma$ is the minimizer of $J_{CB}$, we have $\sum_{j=2}^KjJ'(j\gamma)=-J'(\gamma)$ and thus $\sum_{j=2}^Kc_j=1$.

Next, we show that $J_j$, $j=1,\dots,K$ satisfy (vi)--(viii) with $c_j$ given by \eqref{cjlj} and $\gamma$ given by \eqref{gammaK}. For this, we fix $j\in\{2,\dots,K\}$.
\begin{itemize}
 \item \textit{Argument for} (vi). Consider $z\leq \delta_1$. Since $J$ is decreasing on $(0,\delta_1)$ and increasing on $(\delta_1,\infty)$, the minimum problem in \eqref{ass:unique1} admits a minimizer  $\bar z=(\bar z_1,\dots,\bar z_j)$ satisfying $\bar z_i\in(0,\delta_1]$ for all $i=1,\dots,j$. In combination with the strict convexity of $J$ in $(0,\delta_1]$, we obtain that $\bar z_i=z$ for all $i=1,\dots,j$. Hence, \eqref{ass:unique1} is true with $\e=\delta_1-\gamma>0$, see \eqref{gammaK}. Next, we show that $\gamma$ is the unique minimizer of $J_{0,j}$. Since $J_{0,j}(z)\geq J(jz)+c_jJ(\delta_1)\geq J_{0,j}(\delta_1)$ for $z\geq \delta_1$ it suffices to consider $z\leq\delta_1$ in order to find the minimum of $J_{0,j}$. We already showed $J_{0,j}(z)=J(jz)+c_jJ(z)=\psi_j(z)$ for $z\leq\delta_1$ and thus
$$J_{0,j}(z)=\psi_j(z)=\frac{k_1}{z^{12}}\left(\frac{1}{j^{12}}+c_j\right)-\frac{k_2}{z^{6}}\left(\frac{1}{j^{6}}+c_j\right)\quad\mbox{for $z\leq \delta_1$.}$$
Hence, $\psi_j$ is again a Lennard-Jones potential with a single critical point which is a minimum. Since $c_j$ is defined such that $jJ'(\gamma)+c_jJ'(\gamma)=0$, we deduce that $\gamma$ is the unique minimizer of $\psi_j$ and since $\gamma<\delta_1$ also of $J_{0,j}$.
\item \textit{Argument for} (vii). Let $z$ and $z_s$ be such that $jz=\sum_{s=1}^jz_s$. A Taylor expansion yields 
\begin{align*}
 \sum_{s=1}^jJ(z_s)=&jJ(z)+J'(z)\sum_{s=1}^j(z_s-z)+\sum_{s=1}^j\int_0^1(1-t)J''(z+t(z_s-z))(z_s-z)^2\,dt.
\end{align*}
The second term on the right-hand side vanishes since $\sum_{s=1}^jz_s=jz$. For $\eta>0$ sufficiently small, e.g.\ $\eta<|\gamma-\delta_1|$, we have $J''(z+t(z_s-z))\geq \inf_{z\leq\delta_1} J''(z)>0$ for all $t\in[0,1]$ and $s=1,\dots,j$, which proves the assertion.
\item \textit{Argument for} (viii). Let $(z_n)$ be such that $\lim_{n\to\infty}z_n=+\infty$ and
$$\liminf_{z\to\infty}J_{0,j}(z)=\lim_{n\to\infty}J_{0,j}(z_n).$$
For every $\eta>0$ and $n\in\N$, we find $z_n^s$ with $s=1,\dots,j$ such that
$$J_{0,j}(z_n)\geq J(jz_n)+\frac{c_j}{j}\sum_{s=1}^jJ(z_n^s)-\eta\quad\mbox{with}\quad \sum_{s=1}^jz_n^s=jz_n.$$
Since $z_n\to\infty$ and $J_1(z)=+\infty$ for $z\leq0$, there exists $s\in\{1,\dots,j\}$ such that, up to subsequences, $z_n^s\to+\infty$ as $n\to\infty$. Without loss of generality we assume that $s=1$ and from $\lim_{z\to\infty}J(z)=0$, it follows
$$\liminf_{n\to\infty}J_{0,j}(z_n)\geq \frac{c_j}{j}\liminf_{n\to\infty}\sum_{s=2}^jJ(z_n^s)-\eta\geq c_j\frac{j-1}{j}J(\delta_1)-\eta.$$
Since $J(j\delta_1)<0$ for $j=1,\dots,K$ the assertion follows by choosing $\eta=-\frac12J_j(\delta_1)$ and 
\begin{align*}
 c_j\frac{j-1}{j}J(\delta_1)-\eta>c_j\frac{j-1}{j}J(\delta_1)-\eta+\frac12J(j\delta_1)+\frac{c_j}{j}J(\delta_1)= J(j\delta_1)+c_jJ(\delta_1)> \psi_j(\gamma),
\end{align*}
and since $\psi_j(\gamma)=J_{0,j}(\gamma)$, the assertion is proven.
\end{itemize}

Finally, we comment on identity \eqref{psienv}. We already observed that $\psi_j$ are Lennard-Jones potentials with minimizer $\gamma$, and thus the second equality in \eqref{psienv} follows. The first equality is true since one can easily check that $\psi_j^{**}\leq J_{0,j}\leq \psi_j$.
\end{proof}

\begin{remark}
The proof of Proposition~\ref{prop:lj} can be applied almost verbatim also to slightly more general potentials of the form 
 \begin{equation*}
 J_j(z)=J(jz)\quad\mbox{with}\quad J(z)=\frac{k_1}{z^{m}}-\frac{k_2}{z^n},\quad\mbox{for $z>0$ and $J(z)=+\infty$ for $z\leq0$},
\end{equation*}
$k_1,k_2>0$, and $m>n>1$.
\end{remark}

\begin{remark}
 If $J_j$, $j\in\{1,\dots,K\}$ satisfy (i)--(viii) and \eqref{psienv}, then it is easy to see that $\{\gamma\}=\argmin J_{CB}$ and
\begin{equation}\label{jcbpsi}
 J_{CB}^{**}(z)=\sum_{j=2}^KJ_{0,j}^{**}(z)=\sum_{j=2}^K\psi_{j}^{**}(z)=\begin{cases}
                                           J_{CB}(z)&\mbox{if $z\leq \gamma$}\\
					   J_{CB}(\gamma)&\mbox{if $z\geq \gamma$}.
                                          \end{cases}
\end{equation}
\end{remark}

\section{$\Gamma$-limit of $H_n$}\label{sec:Hn}

In this section, we give an explicit expression for the $\Gamma$-limit of discrete energies $H_n$, see \eqref{ene:bulk}, with periodic boundary conditions. More precisely, for fixed $\ell>0$ we define $H_n^{\ell}:\A_n^{\#,\ell}(0,1)\to\R\cup\{+\infty\}$ as
\begin{equation*}
 H_n^{\ell}(u):=\begin{cases}H_n(u)&\mbox{if $u\in \A_n^{\#,\ell}(0,1)$,}\\+\infty&\mbox{else,}\end{cases}
\end{equation*}
where $ \A_n^{\#,\ell}(0,1)$ is given in \eqref{def:anell}. Moreover, we set
\begin{equation*}
 BV^{\ell}(0,1):=\{u\in BV_{\loc}(\R)\,:\,x\mapsto u(x)-\ell x\quad \mbox{is $1$-periodic}\}.
\end{equation*}
\begin{theorem}\label{zero2}
Let $J_j:\R\to(-\infty,+\infty]$ be Borel functions. Assume there exists a convex function $\Psi:\R\to[0,+\infty]$ such that
\begin{equation}\label{ass:Psi:1}
 \lim_{z\to-\infty}\frac{\Psi(z)}{|z|}=+\infty
\end{equation}
and there exist constants $d^1,d^2>0$ such that
\begin{equation}\label{ass:Psi:2}
d^1(\Psi(z)-1)\leq J_j(z)\leq d^2\max\{\Psi(z),|z|\}\mbox{ for all $z\in\R$.}\end{equation}
Moreover, assume that the $J_j$ satisfy the assumptions (i)--(vi) and \eqref{psienv}. Then, for $\ell>0$, the $\Gamma$-limit of the sequence $(H_n^\ell)$ with respect to the $L_{\loc}^1$-topology is given by 
\begin{equation*}
 H^{\ell}(u):=\begin{cases}
            \displaystyle{\int_0^1} J_{CB}^{**}(u'(x))\,dx&\mbox{if $u\in BV^{\ell}(0,1)$, $D^su\geq0$,}\\
	    +\infty&\mbox{else on $L_{\loc}^1(\R)$, }
           \end{cases}
\end{equation*}
where $D^su$ denotes the singular part of the measure $Du$ with respect to the Lebesgue measure.
\end{theorem}

\begin{remark}\label{rem:stella}
(a) As discussed in the introduction, the $\Gamma$-limit of discrete energies as $(H_n)$ with very general interaction potentials is provided in \cite[Theorem~3.4]{BG04}. Specialized to the situation of Theorem~\ref{zero2}, the limiting energy density in \cite[Theorem~3.4]{BG04} reads
\begin{equation*}
 \overline \phi(z):=\inf\{\phi(z_1)+g(z_2):z_1+z_2=z\}\quad\mbox{where}\quad \phi=\Gamma\text{-}\lim\limits_{N\to\infty}\phi_N^{**},\quad g(z)=\begin{cases}
                                                                                  0&\mbox{if $z\geq0$,}\\+\infty &\mbox{if $z<0$},
                                                                                 \end{cases}
\end{equation*}
with
\begin{align*}
 \phi_N(z)=&\min\bigg\{\frac{1}{N}\sum_{j=1}^K\sum_{i=0}^{N-j}J_j\left( \frac{u^{i+j}-u^i}{j}\right):\notag\\
     &\qquad u:\N_0\to\R, u^i=zi\mbox{ if }i\in\{0,\dots,K\}\cup\{N-K,\dots,N\}\bigg\}. 
\end{align*}

For non-convex interaction potentials, such as the Lennard-Jones potentials, one cannot expect a simplification of the asymptotic homogenization formulas $\overline \phi$ and $\phi$ in general. In fact, the assumptions (v) and (vi) are essential in the simplification of $\overline \phi$ and $\phi$. These assumptions follow for instance from the specific convex-concave shape of the Lennard-Jones potentials.

\smallskip

(b) Theorem~\ref{zero2} follows by showing that $J_{CB}^{**}=\overline \phi$ and adjusting the argument of \cite{BG04} to the present boundary conditions. Here, however, we give a direct proof of Theorem~\ref{zero2} which, by appealing to assumptions (i)--(vi), significantly simplifies compared to \cite{BG04}.  
\end{remark}

\bigskip

\begin{proof}[Proof of Theorem~\ref{zero2}]

{\bf Liminf inequality.} Let $(u_n)\subset W_{\loc}^{1,\infty}(\R)$ be such that $\sup_n H_n^\ell(u_n)<\infty$ and $u_n\to u$ in $L^1_{\loc}$ for some $u\in L^1_{\loc}(\R)$. The growth condition at $-\infty$ of the potentials $J_j$, cf.~\eqref{ass:Psi:1} and \eqref{ass:Psi:2}, implies that $(u_n')^-:=-(u_n'\wedge0)$, where $a\wedge b:=\min\{a,b\}$, satisfies $\sup_n \|(u_n')^-\|_{L^1(0,1)}<\infty$. Combining this with the periodicity of the map $x\mapsto u_n(x)-\ell x$ and $u_n\to u$ in $L^1_{\loc}(\R)$, we obtain $\sup_n\|u_n\|_{W^{1,1}(I)}<\infty$ for every bounded interval $I\subset\R$. Since bounded sequences in $W^{1,1}(I)$ are weakly$^*$ compact in $BV(I)$, we obtain, up to subsequences, $u_n\stackrel{*}{\wto} u$ weakly$^*$ in $BV(I)$ for every bounded interval $I\subset \R$. In particular, this implies $u\in BV_{\loc} (\R)$. Moreover, after extracting a further subsequence, we have that $u_n\to u$ pointwise a.e.~and in combination with the periodicity of $x\mapsto u_n(x)-\ell x$ that $u\in BV^{\ell}(0,1)$.

\smallskip

Let us now estimate the energy. By \eqref{lb:hnu}, we have 

\begin{align*}
  H_n^\ell(u_n)
  \geq&\sum_{j=2}^K\sum_{i=0}^{n-1}\lambda_nJ_{0,j}^{**}\left(\frac{u_n^{i+j}-u_n^i}{j\lambda_n}\right)=\sum_{j=2}^K\frac1j\sum_{s=0}^{j-1}\sum_{i\in (s+j\Z)\cap [0,n)}j\lambda_nJ_{0,j}^{**}\left(\frac{u_n^{i+j}-u_n^i}{j\lambda_n}\right).
\end{align*}
Fix $\rho\in(0,\frac14)$. For all $n$ sufficiently large, we obtain from $J_{0,j}\geq J_{0,j}(\gamma)$ that 
\begin{equation*}
 \sum_{i\in (s+j\Z)\cap [0,n)}j\lambda_nJ_{0,j}^{**}\left(\frac{u_n^{i+j}-u_n^i}{j\lambda_n}\right)\geq \int_\rho^1 J_{0,j}^{**}({u_{n,j}^s}'(x))\,dx+2\rho (J_{0,j}(\gamma)\wedge0),
\end{equation*}
where $u_{n,j}^s$ denotes the piecewise affine interpolation of $u_n$ with respect to the lattice $\lambda_n (s+j\Z)$, i.e.
\begin{equation}\label{def:vnjs}
 u_{n,j}^s(t):=u_n^{s+ji}+\frac{t-(s+ji)\lambda_n}{j\lambda_n}(u_n^{s+j(i+1)}-u_n^{s+ji})\quad\mbox{for $t\in \lambda_ns+\lambda_nj[i,i+1)$ with $i\in\Z$}.
\end{equation}
Using $u_n\stackrel{*}{\wto} u$ weakly$^*$ in $BV(-1,2)$, a straightforward calculation yields that $u_{n,j}^s$ converges weakly$^*$ in $BV(0,1)$ to $u$ for $j\in\{2,\dots,K\}$ and $s\in\{0,\dots,j-1\}$. Hence, a consequence of the superlinear growth at $-\infty$, sublinear growth at $+\infty$, \cite[Theorem 5.2]{AFP} and $\sum_{j=2}^KJ_{0,j}^{**}=J_{CB}^{**}$, see \eqref{jcbpsi}, is
\begin{align*}
 \liminf_{n\to\infty}\sum_{j=2}^K\frac1j\sum_{s=0}^{j-1}\int_\rho^{1-\rho}J_{0,j}^{**}({u_{n,j}^s}'(x))\,dx\geq \sum_{j=2}^K\frac1j\sum_{s=0}^{j-1}\int_\rho^{1-\rho}J_{0,j}^{**}(u'(x))\,dx=\int_\rho^{1-\rho} J_{CB}^{**}(u'(x))\,dx,
\end{align*}
and the constraint $D^su\geq0$ on $(\rho,1-\rho)$. Clearly the liminf inequality follows by letting $\rho$ tend to zero.

\medskip

{\bf Limsup inequality.} \textit{Step~1.} We provide the limsup inequality for a modified discrete energy which does not take the boundary conditions into account and is given by
\begin{equation*}
 \hat H_n(u):=\begin{cases} \displaystyle{\sum_{j=1}^K\sum_{i=0}^{n-j}}J_j\left(\frac{u^{i+j}-u^i}{j\lambda_n}\right)&\mbox{if $u\in \A_n(0,1)$,}\\+\infty&\mbox{else on $L^1(0,1)$,}
 \end{cases}
\end{equation*}
where 
\begin{equation*}
 \A_n(0,1):=\big\{u\in W^{1,\infty}(0,1):u\mbox{ is affine on }(i,i+1)\lambda_n\mbox{ for $i\in\{0,\dots,n\}$}\big\}.
\end{equation*}
We claim that for every $u\in BV(0,1)$ with $D^su\geq 0$, we find a sequence $(u_n)\subset W^{1,\infty}(0,1)$ such that $u_n\to u$ in $L^1(0,1)$ and 
\begin{equation*}
 \limsup_{n\to\infty} \hat H_n(u_n)\leq \int_0^1J_{CB}^{**}(u'(x))\,dx.
\end{equation*}
By density and relaxation arguments it suffices to provide the above inequality for the simpler cases of $u$ linear and of $u$ with a single jump, see e.g.~the proof of \cite[Theorem 3.5]{BG02} for a detailed discussion.

First, we consider functions $u$ with a single jump. Let $u(x) = zx + a\chi(x_0,1]$ with $z \leq \gamma$, $a > 0$, and $x_0\in[0,1]$.  Let $h_n\subset \Z$ be such that $x_0\in \lambda_n[h_n,h_n+1)$ and define $u_n\in \A_n(0,1)$ by $u_n^i=i z$ for all $i\leq h_n$ and $u_n^i=a+iz$ for $i> h_n$. Using (iii), \eqref{def:jcb} and $J_{CB}(z)=J_{CB}^{**}(z)$, we obtain by a direct calculation
\begin{equation*}
 \limsup_{n\to\infty} \hat H_n(u_n)\leq \sum_{j=1}^K J_j(z)=\int_0^1J_{CB}^{**}(u'(x))\,dx.
\end{equation*}
Let us now consider $u(x)=zx$ with $z>\gamma$. We construct a sequence $(u_n)$ converging to $u$ in $L^1(0,1)$ such that $u_n'$ converges to $\gamma$ in measure in $(0,1)$. Let $(N_n)\subset \N$ be such that 
\begin{equation*}
\lim_{n\to\infty} N_n=+\infty\qquad\mbox{and}\qquad \lim_{n\to\infty} \lambda_n N_n=0.
\end{equation*}
Moreover, we define a sequence $(r_n)\subset \N$ by
\begin{equation*}
 r_n:=\sup\{r\in\N\, :\, r N_n\leq n\}.
\end{equation*}
Set $t_n^i=iN_n$ for $i\in\{0,\dots,r_n-1\}$ and $t_n^{r_n}=n$. Define $u_n\in \A_n(0,1)$ such that 
\begin{equation*}
 u_n(x)=\begin{cases}u(\lambda_n t_n^i)+\gamma (x-\lambda_n t_n^i)&\mbox{for $x\in\lambda_n[t_n^i,t_n^{i+1}-1]$ and $i\in\{0,\dots,r_n-2\}$,}\\
 u(x)& \mbox{for $x\in [\lambda_n t_n^{r_n-1},1]$.}
 \end{cases}
\end{equation*}
By the definition of $u_n$ and $u$, we have $\|u_n-u\|_{L^\infty(0,1)}\leq \lambda_n N_n |z-\gamma|$ and thus $u_n\to u$ in $L^1(0,1)$. By construction, we have $u_n'(x)=\gamma$ for all $x\in (0,1)\setminus I_n$ with $I_n=(\cup_{i=1}^{r_n-1}\lambda_n[t_n^i-1,t_n^i])\cup [\lambda_n t_n^{r_n-1},1]$ and $u_n'\geq\gamma$ on $I_n$. Using $\lim_{n\to\infty}|I_n|=0$ and $(0,+\infty)\subset \dom J_j$ (see assumption (i)), we obtain
\begin{equation*}
 \limsup_{n\to\infty} \hat H_n(u_n)\leq \sum_{j=1}^K J_j(\gamma)=\int_0^1J_{CB}^{**}(u'(x))\,dx.
\end{equation*}

\textit{Step~2.} We show that there exists for every $u \in BV^\ell(0,1)$ a sequence $(u_n)$ such that $u_n\to u$ in $L^1(0,1)$ and $\limsup_{n\to\infty}H_n^\ell(u_n)\leq H^\ell(u)$. We follow ideas from \cite[Theorem 4.2]{BDG99}, where the case of nearest neighbor interactions and Dirichlet boundary conditions is considered. \\
Let us first consider functions with a jump at zero: Let $u\in BV^\ell(0,1)$ be such that $H^\ell(u)<\infty$, and $u(0-)<u(0+)$. By the previous step, we find a sequence $(u_n)$ such that $u_n\to u$ in $L^1(0,1)$ and
\begin{equation}\label{rec:without}
\limsup_{n\to\infty} \hat H_n(u_n)=\limsup_{n\to\infty}\sum_{j=1}^K\sum_{i=0}^{n-j}J_j\left(\frac{u_n^{i+j}-u_n^i}{j\lambda_n}\right)\leq \int_0^1J_{CB}^{**}(u'(x))\,dx.
\end{equation}
Next, we introduce a suitable pertubation of $(u_n)$ which takes the periodic boundary condition into account. By passing to a subsequence, it is not restrictive to assume that $u_n\to u$ pointwise a.e.~in $(0,1)$. Hence, for every $\hat \e>0$ there exists $\e\in(0,\hat \e)$ such that $\e,1-\e\notin S_u$, $u_n(\e)\to u(\e)$, and $u_n(1-\e)\to u(1-\e)$. For $\hat \e > 0$ sufficiently small, we deduce from $u(0-) < u(0+)$, \eqref{ass:Psi:1}, \eqref{ass:Psi:2}, $D^su\geq 0$  and the periodicity of $x\mapsto u(x)-\ell x$ that 
$$\tfrac12(u(0-)+u(0+))+2\e \gamma<u(\e),\quad u(1-\e)+ 2\e \gamma<\tfrac12(u(0-)+u(0+)) + \ell =\tfrac12(u(1-)+u(1+)) .$$
Let $(h_n^1),(h_n^2)\subset\N$ be such that $\e\in\lambda_n[h_n^1,h_n^1+1)$ and $1-\e\in\lambda_n(h_n^2-1,h_n^2]$. We define $v_{n,\e}\in\A_n^{\#,\ell}(0,1)$ as the unique function in $\A_n^{\#,\ell}(0,1)$ satisfying
\begin{equation*}
 v_{n,\e}^i=\begin{cases}
 \frac12(u(0-)+u(0+))+i\lambda_n\gamma&\mbox{for $0\leq i< h_n^1$}\\
 u_n(\e)-\frac12\gamma \e&\mbox{for $i=h_n^1$,}\\
 u_n^i&\mbox{for $h_n^1< i <h_n^2$,}\\
 u_n(1-\e)+\frac12\gamma \e&\mbox{for $i=h_n^2$,}\\
 \frac12(u(1-)+u(1+))-\gamma+i\lambda_n\gamma&\mbox{for $h_n^2< i <  n.$}
 \end{cases}
\end{equation*}
We observe that $(v_{n,\e})$ converges in $L_{\loc}^1$ to $u_\e\in BV^\ell(0,1)$, where 
\begin{equation*}
  u_\e(x)=
  \begin{cases} 
  \frac12(u(0-)+u(0+))+\gamma x&\mbox{if $ x\in(0,\e)$,}\\ 
 u(x)&\mbox{if $x\in(\e,1-\e)$,}\\
 \frac12(u(1-)+u(1+))+\gamma(x-1)&\mbox{if $x\in(1-\e,1)$.}
 \end{cases}
\end{equation*}
The construction of $v_{n,\e}$ is such that 
\begin{align}\label{jump:vnis}
 \lim_{n\to\infty}\frac{v_{n,\e}^{h_n^{i}+s}-v_{n,\e}^{h_n^{i}+s-1}}{\lambda_n}=+\infty\qquad\mbox{for $i\in\{1,2\}$ and $s\in\{0,1\}$.}
\end{align}
Combining \eqref{rec:without}--\eqref{jump:vnis}, $\gamma>0$, and $(0,+\infty)\subset \dom J_j$ for all $j\in\{1,\dots,K\}$ (see assumption (i)), we obtain that 
\begin{align}\label{eq:limsupeps}
\limsup_{n\to\infty} H_n^\ell(v_n)&\leq\limsup_{n\to\infty} \sum_{j=1}^K\sum_{i=0}^{n-1}\lambda_nJ_j\left(\frac{v_{n,\e}^{i+j}-v_{n,\e}^i}{j\lambda_n}\right)\notag\\
&\leq \int_\e^{1-\e} J_{CB}^{**}(u'(x))\,dx+2\e J_{CB}(\gamma)=H^\ell(u_\e).
\end{align}
From \eqref{eq:limsupeps} the existence of a recovery sequence for $u$ follows by the lower semi-continuity of the $\Gamma$-$\limsup$, see e.g.~\cite[Remark~1.29]{B02}.

Finally, we consider $u\in BV^\ell(0,1)$ such that $H^\ell(u)<\infty$ and $u(0)=u(0-)=u(0+)$. As it is discussed in \cite[p.~40]{BDG99}, we find a suitable approximation of $u$ by functions with a positive jump in $0$, i.e.~a sequence $(u_j)$ satisfying $u_j\to u$ in $L_{\loc}^1(\R)$ such that $\int_0^1J_{CB}^{**}(u_j'(x))\,dx\to\int_0^1J_{CB}^{**}(u'(x))\,dx$ and $u_j(0+)>u_j(0-)=u(0)$. By the previous considerations, we obtain a recovery sequence for every $u_j$ and the existence of a recovery sequence for $u$ follows again by the lower semi-continuity of the $\Gamma$-$\limsup$.

\end{proof}

\section{$\Gamma$-limit of $E_n^\ell$}\label{sec:En}

In this section, we derive the $\Gamma$--limit of the sequence $(E_n^\ell)$, defined in \eqref{ene:rescaledbc}. For this, it is useful to rewrite the energy $E_n^\ell(v)$. For every $v_n\in\A_n^{\#,\ell}(0,1)$, we define for $j\in\{2,\dots,K\}$
\begin{equation}\label{def:zeta}
 \zeta_{j,n}^i:=J_j\left(\frac{v_n^{i+j}-v_n^i}{j\sqrt{\lambda_n}}+\gamma\right)+\frac{c_j}{j}\sum_{s=i}^{i+j-1}J_1\left(\frac{v_n^{s+1}-v_n^s}{\sqrt{\lambda_n}}+\gamma\right)-J_{0,j}(\gamma).
\end{equation}
Using \eqref{rewJ1} and  $J_{CB}(\gamma)=\sum_{j=2}^KJ_{0,j}(\gamma)$, we can write $E_n^\ell(v_n)$ for any $v_n\in\A_n^{\#,\ell}(0,1)$ as
\begin{equation}\label{ene:rew}
 E_n^\ell(v_n)=\sum_{j=2}^K\sum_{i=0}^{n-1}\zeta_{j,n}^i.
\end{equation}
By the definition of $J_{0,j}$ and $\gamma$, we have $\zeta_{j,n}^i\geq J_{0,j}(\gamma+\frac{v_n^{i+j}-v_n^i}{j\sqrt{\lambda_n}})-J_{0,j}(\gamma)\geq0$ for $j\in\{2,\dots,K\}$ and $i\in\{0,\dots,n-j\}$. The following lemma yields a sharper lower bound on $\zeta_{n,j}^i$; it is inspired by \cite[Remark 4]{BLO} and will be applied in the proof of Theorem~\ref{Th1}.
\begin{lemma}\label{lemma:CB}
Suppose that $J_j$, $j=1,\dots,K$ satisfy the assumptions (i)--(viii). For $\eta_1>0$ sufficiently  small there exists a constant $C_1>0$ such that for all $j\in\{2,\dots,K\}$
\begin{equation}\label{ineq:L:1}
 J_j\left(\sum_{s=1}^j\frac{z_s}{j}\right)+\frac{c_j}{j}\sum_{s=1}^jJ_1(z_s)-J_{0,j}(\gamma)\geq C_1\sum_{s=1}^j(z_s-\gamma)^2
\end{equation}
if $\sum_{s=1}^j|z_s-\gamma|\leq \eta_1$
\end{lemma}

\begin{proof}
Fix $j\in\{2,\dots,K\}$. If $\sum_{s=1}^jz_s=j\gamma$, the claim follows from assumption (vii). Let $\e,\eta>0$ denote the same constants as in assumption (vi) and (vii). By \eqref{def:psij}, we have $J_{0,j}=\psi_j=J_j+c_jJ_1$ on $(-\infty,\gamma+\e]$. Moreover, recall that $\psi_j\in C^2(0,+\infty)$, $\gamma>0$ and $\psi_j''(\gamma)>0$. Hence, we find $\eta_1\in(0,\e)$ and $\delta>0$ such that $\sum_{s=1}^j|z_s-\gamma|\leq \eta_1$ implies $\sum_{s=1}^j|z_s-z|+|z-\gamma|\leq \eta$, where $z=\frac1j \sum_{s=1}^jz_s$ and $\psi_j''\geq\delta$ on $[\gamma-\eta_1,\gamma+\eta_1]$. 

Let us now show \eqref{ineq:L:1} whenever $\sum_{s=1}^j|z_s-\gamma|\leq \eta_1$. Assume by contradiction that there exist $\hat z_s$, $s=1,\dots,j$ satisfying $\sum_{s=1}^j|\hat z_s-\gamma|\leq\eta_1$ and for all $N>2$ 
$$J_j(\hat z)+\frac{c_j}{j}\sum_{s=1}^jJ_1(\hat z_s)-J_{0,j}(\gamma)\leq \frac{C}{N}\sum_{s=1}^j(\hat z_s-\gamma)^2,$$
where $C$ is the same constant as in \eqref{cb}, and $\hat z:=\frac1j\sum_{s=1}^j\hat z_s$. By the choice of $\eta_1$, we have $\sum_{s=1}^j|\hat z_s-\hat z|+|\hat z-\gamma|\leq \eta$ and thus by \eqref{cb} it holds
\begin{align*}
 J_j(\hat z)+\frac{c_j}{j}\sum_{s=1}^jJ_1(\hat z_s)-J_{0,j}(\gamma)\leq& \frac{C}{N}\sum_{s=1}^j(\hat z_s-\gamma)^2\leq\frac{2C}{N}\sum_{s=1}^j(\hat z_s-\hat z)^2+\frac{2Cj}{N}(\hat z-\gamma)^2\\
\leq&\frac{2}{N}\left(J_j(\hat z)+\frac{c_j}{j}\sum_{s=1}^jJ_1(\hat z_s)-\psi_j(\hat z)\right)+\frac{2Cj}{N}(\hat z-\gamma)^2.
\end{align*}
Using $\psi_j(\hat z)\geq\psi_j(\gamma)$ and $J_{0,j}(\hat z)=\psi_j(\hat z)$ (since $\hat z\leq \gamma+\e$ and \eqref{def:psij}), we obtain 
\begin{align*}
 \psi_j(\hat z)-\psi_j(\gamma)\leq J_j(\hat z)+\frac{c_j}{j}\sum_{s=1}^jJ_1(\hat z_s)-\psi_j(\gamma)\leq \frac{2jC}{N-2}(\hat z-\gamma)^2.
\end{align*}
Clearly, this is, for $N$ sufficiently large, a contradiction to
$$\psi_j(\hat z)-\psi_j(\gamma)=\int_0^1\psi_j''(\gamma + s(\hat z-\gamma))(1-s)(\hat z-\gamma)^2\,ds\geq\frac12\delta(\hat z-\gamma)^2,$$
where we use $\psi_j'(\gamma)=0$.

\end{proof}

We are now in the position to prove the main result of this section which is a $\Gamma$-convergence result for the functionals $(E_n^\ell)$. 
\begin{theorem}\label{Th1}
Suppose that $J_j$, $j=1,\dots,K$ satisfy the assumptions (i)--(viii). Then the sequence $(E_{n}^\ell)$ $\Gamma$--converges with respect to the $L_{\loc}^1$--topology to the functional $E^\ell$ defined, on piecewise-$H^1$ functions $v$ such that $x\mapsto v(x)-\ell x$ is $1$-periodic, by
\begin{equation*}
 E^\ell(v)=\begin{cases}
               \alpha\displaystyle{\int_0^1}v'(x)^2\,dx+\beta(\#S_v\cap[0,1))&\mbox{if $[v](x)>0$ on $S_v$,}\\
	    +\infty&\mbox{else,}
              \end{cases}
\end{equation*}
where $\alpha:=\frac12J_{CB}''(\gamma)$. Further, the boundary layer energy due to a jump of $v$ is given by
\begin{equation}\label{def:bij}
 \beta:=2B(\gamma)-\sum_{j=1}^KjJ_j(\gamma)
\end{equation}
with
\begin{align}\label{Bgamma}
 B(\gamma):=&\inf_{N\in\mathbb{N}_0}\min\Bigg\{\sum_{i\geq0}\bigg\{\sum_{j=1}^KJ_j\left(\frac{u^{i+j}-u^i}{j}\right)-J_{CB}(\gamma)\bigg\}:\notag\\&\hspace*{2cm}u:\mathbb N_0\to \mathbb R,u^{0}=0,u^{i+1}-u^{i}=\gamma\mbox{ if $i\geq N$}\Bigg\}.
\end{align}
Moreover, if $\ell>0$ it holds
\begin{equation*}
 \lim_{n\to\infty}\inf_v E_n^\ell(v)=\min_v E^\ell(v)=\min\{\alpha \ell^2,\beta\}.
\end{equation*}
\end{theorem}

The following equivalent formulation of the boundary layer energy will  be convenient for the proof of Theorem~\ref{Th1}:
\begin{lemma}\label{prop:beta}
Let $J_j$ satisfy the assumptions of Theorem \ref{Th1}. Then it holds
$$\beta=2\widetilde B(\gamma)-\sum_{j=2}^KjJ_{0,j}(\gamma),$$
where $\beta$ is given in \eqref{def:bij} and $\widetilde B(\gamma)$ is defined by
\begin{align}\label{Bgamma2}
\widetilde B(\gamma):=&\inf_{k\in\mathbb{N}_0}\min\bigg\{\sum_{j=2}^Kc_j\sum_{s=1}^{j-1}\frac{j-s}{j}J_1\left(u^{s}-u^{s-1}\right)+\sum_{j=2}^K\sum_{i\geq0}\bigg\{J_j\left(\frac{u^{i+j}-u^i}{j}\right)\notag\\
&\hspace*{2cm}+\frac{c_j}{j} \sum_{s=i}^{i+j-1}J_1\left(u^{s+1}-u^s\right)-J_{0,j}(\gamma)\bigg\}:\notag\\
&\hspace*{2cm} u:\mathbb N_0\to \mathbb R,u^{0}=0,u^{i+1}-u^{i}=\gamma\mbox{ if $i\geq k$}\bigg\}.
\end{align}
\end{lemma}
We postpone the calculations regarding Lemma~\ref{prop:beta} and directly turn to the proof of Theorem~\ref{Th1}.

\begin{proof}[Proof of Theorem \ref{Th1}]

\textbf{Coerciveness.} Let $(v_n)\subset W_{\loc}^{1,\infty}(\R)$ be such that $\sup_n E_n^\ell(v_n)<+\infty$. By assumption (vi) and Lemma \ref{lemma:CB}, there exist constants $K_1,~K_2>0$ such that for all $i\in\Z$ it holds
\begin{align}\label{est:zeta}
 \zeta_{n,j}^i&\geq \left\{K_1\sum_{s=i}^{i+j-1}\left(\frac{v_n^{s+1}-v_n^s}{\sqrt{\lambda_n}}\right)^2\right\}\wedge K_2,
\end{align}
where $a\wedge b:=\min \{a,b\}$ and $\zeta_{n,j}^i$ is given in \eqref{def:zeta}. Hence, \eqref{ene:rew} and \eqref{est:zeta} yield
\begin{align}\label{ene:coer}
 E_n^\ell(v_n)\geq&\sum_{j=2}^K\sum_{i=0}^{n-1}\left\{\lambda_nK_1\sum_{s=i}^{i+j-1}\left(\frac{v_n^{s+1}-v_n^s}{\lambda_n}\right)^2\right\}\wedge K_2\geq\sum_{i=0}^{n-1}\left(\lambda_n K_1\left(\frac{v_n^{i+1}-v_n^i}{\lambda_n}\right)^2\wedge K_2\right).
\end{align}
The discrete energy on the right-hand side of \eqref{ene:coer} is well studied, see e.g.~\cite[Remark~9]{BLO}. In particular, we can conclude from \eqref{ene:coer} that if $(v_n)$ is bounded in $L^1(0,1)$ then $(v_n)$ is compact in $L^1(0,1)$ and there exists a finite set $S\subset [0,1]$ such that $(v_n)$ is locally weakly compact in $H^1((0,1)\setminus S)$. 

\smallskip

Let us remark that $\sup_n E_n^\ell(v_n)<+\infty$ implies $\sup_n\|v_n'\|_{L^{1}(0,1)}<+\infty$. To show this, we combine \eqref{ene:coer} with the growth conditions of $J_j$ at $-\infty$, see hypothesis (ii). For every $n\in N$, we set
$$I_n^-:=\left\{i\in\{0,\dots,n-1\}:v_n^{i+1}<v_n^i\right\},\quad I_n^{--}:=\left\{i\in I_n^-:\lambda_nK_1\left(\frac{v_n^{i+1}-v_n^i}{\lambda_n}\right)^2\geq K_2\right\}.$$
The estimate \eqref{ene:coer} implies $I^{--}:=\sup_n\#I_n^{--}<+\infty$. Moreover, the equiboundedness of the energy, assumption (ii), $\zeta_{n,j}^i\geq0$, and the fact that $J_j$ is bounded from below for $j\in\{1,\dots,K\}$ imply that there exists a constant $M\in\R$ such that
\begin{equation}\label{est:linftyvn-}
 \gamma+\frac{v_n^{i+1}-v_n^i}{\sqrt{\lambda_n}}\geq M.
\end{equation}
Hence, using H\"older's inequality, $\#I_n^{-}\leq n$, \eqref{ene:coer} and \eqref{est:linftyvn-}, we have for $(v_n')^-:=-(v_n'\wedge0)$ that
\begin{align*}
 \|(v_n')^-\|_{L^1(0,1)}\leq& \sum_{i\in I_n^-\setminus I_n^{--}}\lambda_n\left|\frac{v_n^{i+1}-v_n^i}{\lambda_n}\right|+\sum_{i\in I_n^{--}}\sqrt{\lambda_n}\left|\frac{v_n^{i+1}-v_n^i}{\sqrt{\lambda_n}}\right|\\
\leq& \left(\frac{1}{K_1}E_n^\ell(v_n)\right)^\frac12+1+\sqrt{\lambda_n}\#I^{--}|M-\gamma|.
\end{align*}
Thus there exists $C>0$ such that $\sup_{n}\|(v_n')^-\|_{L^1(0,1)}<C$ and, appealing to $v_n(1)-v_n(0)=\ell$ by the $1$-periodicity of $x\mapsto v_n(x)-\ell x$, we obtain that $(v_n')$ is equibounded in $L^1(0,1)$.

\bigskip

\textbf{Liminf inequality.} Let $(v_n)\subset W_{\loc}^{1,\infty}(\R)$ be such that $\sup_n E_n^\ell(v_n)<+\infty$ and $v_n\to v$ in $L^1(0,1)$. By the previous step, we can assume that $v$ is a piecewise $H^1$-function satisfying $x\mapsto v(x)-\ell x$ is $1$-periodic, and there exists a finite set $S=\{x_1,\dots,x_N\}$ such that $v_n\wto v$ locally weakly in $H^1((0,1)\setminus S)$.  For simplicity, we assume that $v$ has a single jump and without loss of generality we set $S_v=\{\frac12\}$.

\smallskip

\textit{Step~1.} We estimate the elastic energy and show non-existence of negative jumps.

\smallskip

To this end, we adjust arguments given in \cite[Proof of Theorem 4, Step 2]{BLO} to the present situation. In particular, we show that the maps $z\mapsto J_{0,j}(z)-J_{0,j}(\gamma)$ can be estimated from below by certain truncated quadratic potentials, similar to \cite[eq.~(111)]{BLO}. This allows to apply $\Gamma$-convergence results for truncated quadratic potentials, see \cite[Section 8.3]{B02}. 

The assumptions (ii), (vi) and (viii) imply
\begin{equation}\label{lim:j0j}
 \liminf_{z\to+\infty} J_{0,j}(z)>J_{0,j}(\gamma),\quad \liminf_{z\to-\infty}J_{0,j}(z)=+\infty.
\end{equation}
Combining \eqref{lim:j0j} and the fact that $\gamma$ is the unique minimizer of $J_{0,j}$, we find for each $j\in\{2,\dots,K\}$ constants $C_{1,j},C_{2,j},C_{3,j}>0$ such that 
\begin{equation}\label{j0jlb}
 J_{0,j}(z)-J_{0,j}(\gamma)\geq \Psi_j(z-\gamma):=\begin{cases}C_{1,j}(z-\gamma)^2\wedge C_{2,j}&\mbox{if $z\geq \gamma$,}\\
            C_{1,j}(z-\gamma)^2\wedge C_{3,j}&\mbox{if $z\leq \gamma$.}                                        
                                                   \end{cases}
\end{equation}
Since $J_{0,j}(z)=\psi_j(z)$ for $z\leq \gamma+\e$, see hypothesis (vi), we obtain
\begin{equation}\label{sup:c1}
 \sup\left\{C_{1,j}:\mbox{\eqref{j0jlb} holds }\right\}=\frac12 \psi_j''(\gamma)\qquad\mbox{for all $j\in\{2,\dots,K\}$}.
\end{equation}
Moreover, \eqref{lim:j0j} implies 
\begin{equation}\label{sup:c3}
 \sup\left\{C_{3,j}:\mbox{\eqref{j0jlb} holds for some $C_{1,j}$ and $C_{2,j}$}\right\}=+\infty\qquad\mbox{for all $j\in\{2,\dots,K\}$}.
\end{equation}
Using \eqref{j0jlb}, we have the following estimate
\begin{align}\label{ene:lbpsi}
 E_n^\ell(v_n)=& \sum_{j=2}^K\sum_{i=0}^{n-1}\zeta_{j,n}^i\geq\sum_{j=2}^K\sum_{i=0}^{n-1}\left\{J_{0,j}\left(\gamma+\frac{v_n^{i+j}-v_n^i}{j\sqrt{\lambda_n}}\right)-J_{0,j}(\gamma)\right\}\notag\\
 \geq& \sum_{j=2}^K\sum_{s=0}^{j-1}\sum_{i\in (s+j\Z)\cap [0,n)}\Psi_j\left(\frac{v_n^{i+j}-v_n^i}{j\sqrt{\lambda_n}}\right).
\end{align}
As mentioned above, discrete energies with potentials of the type $\Psi_j$ are well-studied, see e.g.~\cite[Section 8.3]{B02} or \cite[Remark 9]{BLO}. In particular, we have for every $j\in\{2,\dots,K\}$ and $s\in\{0,\dots,j-1\}$ that
\begin{align}\label{lim:lbpsi}
 &\liminf_{n\to\infty}\sum_{i\in (s+j\Z)\cap  [0,n)}\Psi_j\left(\frac{v_n^{i+j}-v_n^i}{j\sqrt{\lambda_n}}\right)\notag\\
 &\qquad \geq \frac{C_{1,j}}{j}\int_0^1v'(x)^2\,dx+C_{2,j}\#\{t\in S_v:[v](t)>0\}+C_{3,j}\#\{t\in S_v:[v](t)<0\}.
\end{align}
Here, we use that the piecewise affine interpolations $v_{n,j}^s$ of $v_n$ with respect to the lattice $\lambda_n (s+j\Z)$, cf.~\eqref{def:vnjs}, satisfy $v_{n,j}^s\to v$ in $L^1_{\loc}(\R)$. Using \eqref{sup:c1}--\eqref{lim:lbpsi}, and $\frac12\sum_{j=2}^K\psi_j''(\gamma)=\frac12 J_{CB}''(\gamma)=\alpha$, we get existence of $C_2\geq0$ such that 
\begin{equation}\label{est:liminfen0}
 \liminf_{n\to\infty} E_n^\ell(v_n)=\liminf_{n\to\infty}\sum_{j=2}^K\sum_{i=0}^{n-1}\zeta_{j,n}^i\geq \alpha\int_0^1v'(x)^2\,dx+C_2\# S_v
\end{equation}
with $[v](t)>0$ on $S_v$, and $+\infty$ else.  

For later usage, we state an estimate involving only terms which contribute to the elastic energy and are sufficiently far away from the jump. For given $\rho>0$, let $(\hat k_n^1),(\hat k_n^2)\subset\N$ be such that $\frac12-2\rho\in\lambda_n[\hat k_n^1,\hat k_n^1+1)$ and $\frac12+2\rho\in\lambda_n[\hat k_n^2,\hat k_n^2+1)$. A similar calculations as for \eqref{est:liminfen0} yields
\begin{equation}\label{est:liminfen1}
\liminf_{n\to\infty}\sum_{j=2}^K\left\{\sum_{i=0}^{\hat k_n^1}\zeta_{j,n}^i+\sum_{i=\hat k_n^2}^{n-1}\zeta_{j,n}^i\right\}\geq \alpha\int_{(0,1)\setminus (\frac12-3\rho,\frac12+3\rho)}v'(x)^2\,dx.
\end{equation}

\medskip

\textit{Step~2.} We estimate the jump energy. 

\smallskip

Recall that $S_v=\{\frac12\}$. By the above consideration leading to \eqref{est:liminfen0} the jump has to be positive. Let $\rho\in(0,\frac18)$ be sufficiently small such that $\{\frac12\}=(\frac12-\rho,\frac12+\rho)\cap S$. We claim existence of $(k_n^1),(k_n^2)\subset \N$ such that $\frac12-\rho\leq\lambda_n(k_n^1+s)\leq \frac12-\frac{\rho}{2}$ and $\frac12+\frac{\rho}{2}\leq \lambda_n(k_n^2+s)\leq \frac12+\rho$ for $s=1,\dots,K$ with
\begin{equation}\label{lim:kn}
 \lim_{n\to\infty}\frac{v_n^{k_n^1+s+1}-v_n^{k_n^1+s}}{\sqrt{\lambda_n}}=0,\quad \lim_{n\to\infty}\frac{v_n^{k_n^2+s+1}-v_n^{k_n^2+s}}{\sqrt{\lambda_n}}=0\quad\mbox{for $s=1,\dots,K-1$.}
\end{equation}
We argue by contradiction: suppose that there exists $c>0$ such that for all $i\in\N$ satisfying $\frac12-\rho\leq \lambda_n (i+s)\leq \frac12-\frac{\rho}{2}$ with $s\in\{1,\dots,K-1\}$ there exists an $\hat s\in\{1,\dots,K-1\}$ such that $|\frac{v_n^{i+\hat s+1}-v_n^{i+\hat s}}{\sqrt{\lambda_n}}|\geq c$. Let $i_n^\rho,j_n^\rho\subset\N$  be such such that $\frac12-\rho\in(i_n^\rho-1,i_n^\rho]\lambda_n$ and $\frac12-\frac{\rho}{2}\in(j_n^\rho,j_n^\rho+1]\lambda_n$. We have by \eqref{ene:coer}
$$E_n^\ell(v_n)\geq \sum_{i=i_n^\rho+1}^{j_n^\rho-K}\zeta_{n,K}^i\geq \sum_{i=i_n^\rho+1}^{j_n^\rho-K}K_1c^2\wedge K_2\geq K_1c^2\wedge K_2(j_n^\rho-i_n^\rho-K)\to+\infty$$
as $n\to \infty$, which contradicts $\sup_nE_n^\ell(v_n)<+\infty$. This implies the existence of $(k_n^1)$ with the above properties, and the existence of $(k_n^2)$ follows with the same argument.

We claim that
\begin{equation}\label{est:lbjump}
 \liminf_{n\to\infty}\sum_{j=2}^K\sum_{i=k_n^1+1}^{k_n^2}\zeta_{j,n}^i\geq 2\widetilde B(\gamma)-\sum_{j=2}^KjJ_{0,j}(\gamma),
\end{equation}
where $\widetilde B(\gamma)$ is given in \eqref{Bgamma2}. Notice, that \eqref{est:lbjump} finishes the proof of the $\liminf$ inequality. Indeed, a combination of \eqref{est:liminfen1}, where $\hat k_n^1<k_n^1$ and $k_n^2<\hat k_n^2$ for $n$ sufficiently large, with \eqref{est:lbjump} and Lemma~\ref{prop:beta} implies
\begin{equation*}
 \liminf_{n\to\infty} E_n^\ell(v_n)=\liminf_{n\to\infty}\sum_{j=2}^K\sum_{i=0}^{n-1}\zeta_{j,n}^i\geq \alpha\int_{(0,1)\setminus (\frac12-3\rho,\frac12+3\rho)}v'(x)^2\,dx+\beta,
\end{equation*}
and the $\liminf$ inequality follows by letting $\rho$ tend to zero.

Let us prove \eqref{est:lbjump}. From $v_n\to v$ in $L^1(0,1)$ and $\frac12\in S_v$, we deduce that there exists $(h_n)\subset\N$ with $\lambda_n h_n\to \frac12$ such that
$$\lim_{n\to\infty}\frac{v_n^{h_n+1}-v_n^{h_n}}{\sqrt{\lambda_n}}=+\infty.$$
Indeed, otherwise $v_n'$ would be equibounded in $L^2$ in a neighborhood of $\frac12$.\\
Since $\lim_{z\to\infty}J_j(z)=0$ for $j=1,\dots,K$, we conclude that some terms in $\zeta_{j,n}^{h_n-s}$ for $s=0,\dots,j-1$ and $j=2,\dots,K$ vanish as $n$ tends to infinity. We collect them in the function $r_1(n)$ defined by
$$r_1(n):=\sum_{j=1}^K\sum_{s=-j+1}^0J_j\left(\gamma+\frac{v_n^{h_n+j+s}-v_n^{h_n+s}}{j\sqrt{\lambda_n}}\right)=\sum_{j=1}^K\sum_{s=h_n-j+1}^{h_n}J_j\left(\gamma+\frac{v_n^{s+j}-v_n^s}{j\sqrt{\lambda_n}}\right).$$
It will be useful to rewrite the terms which involve $v_n^{h_n+1}-v_n^{h_n}$ as follows:
\begin{align*}
 &\sum_{j=2}^K\sum_{i=h_n-j+1}^{h_n}\zeta_{j,n}^i\\
&=\sum_{j=2}^Kc_j\sum_{s=1}^j\frac{j-s}{j}\left(J_1\left(\gamma+\frac{v_n^{h_n-s+1}-v_n^{h_n-s}}{\sqrt{\lambda_n}}\right)+J_1\left(\gamma+\frac{v_n^{h_n+s+1}-v_n^{h_n+s}}{\sqrt{\lambda_n}}\right)\right)\\
&\quad-\sum_{j=2}^KjJ_{0,j}(\gamma)+r_1(n).
\end{align*}
%
%
Hence,
\begin{align}\label{rev:inf:jump}
 \sum_{j=2}^K\sum_{i=k_n^1+1}^{k_n^2}\zeta_{j,n}^i=&\sum_{j=2}^K\bigg\{\sum_{i=k_n^1+1}^{h_n-j}\zeta_{j,n}^i+c_j\sum_{s=1}^{j-1}\frac{j-s}{j}J_1\left(\gamma+\frac{v_n^{h_n-s+1}-v_n^{h_n-s}}{\sqrt{\lambda_n}}\right)\notag\\
&\qquad+\sum_{i=h_n+1}^{k_n^2}\zeta_{j,n}^i+c_j\sum_{s=1}^{j-1}\frac{j-s}{j}J_1\left(\gamma+\frac{v_n^{h_n+s+1}-v_n^{h_n+s}}{\sqrt{\lambda_n}}\right)\bigg\}\notag\\
&-\sum_{j=2}^KjJ_{0,j}(\gamma)+r_1(n).
\end{align}
Thus it remains to prove that
\begin{align}
& \sum_{j=2}^K\bigg\{\sum_{i=k_n^1+1}^{h_n-j}\zeta_{j,n}^i+c_j\sum_{s=1}^{j-1}\frac{j-s}{j}J_1\left(\gamma+\frac{v_n^{h_n-s+1}-v_n^{h_n-s}}{\sqrt{\lambda_n}}\right)\bigg\}\geq \widetilde B(\gamma)-r_2(n)\label{inf:int:1}\\
&\sum_{j=2}^K\bigg\{\sum_{i=h_n+1}^{k_n^2}\zeta_{j,n}^i+c_j\sum_{s=1}^{j-1}\frac{j-s}{j} J_1\left(\gamma+\frac{v_n^{h_n+s+1}-v_n^{h_n+s}}{\sqrt{\lambda_n}}\right)\bigg\}\geq \widetilde B(\gamma)-r_3(n)\label{inf:int:2}
\end{align}
with $r_2(n),r_3(n)\to0$ as $n\to\infty$. Let us prove inequality \eqref{inf:int:1}. Therefore, we define for $i\geq0$
\begin{equation}
 \tilde v_n^i=\begin{cases}
               \gamma i+\frac{v_n^{h_n}-v_h^{h_n-i}}{\sqrt{\lambda_n}}&\mbox{if $0\leq i\leq h_n-k_n^1-1$,}\\
	       \gamma i+\frac{v_n^{h_n}-v_n^{k_n^1+1}}{\sqrt{\lambda_n}}&\mbox{if $i\geq h_n-k_n^1-1$.}
              \end{cases}
\end{equation}
Now we rewrite the left-hand side in \eqref{inf:int:1} in terms of $\tilde v_n$: 
\begin{align*}
 &\sum_{j=2}^K\bigg\{\sum_{i=k_n^1+1}^{h_n-j}\zeta_{j,n}^i+c_j \sum_{s=1}^{j-1}\frac{j-s}{j}J_1\left(\gamma+\frac{v_n^{h_n-s+1}-v_n^{h_n-s}}{\sqrt{\lambda_n}}\right)\bigg\}\\
&=\sum_{j=2}^Kc_j\sum_{s=1}^j\frac{j-s}{j}J_1(\tilde v_n^s-\tilde v_n^{s-1})+\sum_{j=2}^K\sum_{i\geq0}\bigg\{J_j\left(\frac{\tilde v_n^{i+j}-\tilde v_n^j}{j}\right)\\&\quad+\frac{c_j}{j}\sum_{s=i}^{i+j-1}J_1(\tilde v_n^{s+1}-v_n^s)-J_{0,j}(\gamma)\bigg\}-r_2(n)
\end{align*}
where
\begin{align*}
 r_2(n):=\sum_{j=2}^K\sum_{i=h_n-k_n^1-j}^{h_n-k_n^1-2}\bigg\{J_j\left(\frac{\tilde v_n^{i+j}-\tilde v_n^j}{j}\right)+\frac{c_j}{j}\sum_{s=i}^{i+j-1}J_1(\tilde v_n^{s+1}-v_n^s)-J_{0,j}(\gamma)\bigg\}.
\end{align*}
Indeed, the definition of $\tilde v_n$ yields $\tilde v_n^{i+1}-\tilde v_n^i=\gamma$ for $i\geq h_n-k_n^1-1$ and thus the terms in the infinite sum over $i\geq h_n-k_n^1-1$ are $J_j(\gamma)+c_jJ_1(\gamma)-J_{0,j}(\gamma)=0$, cf.\ (vi). Furthermore, we have by the definition of $(\tilde v_n)$ and \eqref{lim:kn}:
$$\lim_{n\to\infty}\left(\tilde v_n^{h_n-k_n^1-K+s}-\tilde v_n^{h_n-k_n^1-K+s-1}\right)=\gamma+\lim_{n\to\infty}\frac{v_n^{k_n^1+1+K-s}-v_n^{k_n^1+K-s}}{\sqrt{\lambda_n}}=\gamma$$ 
for $s\in\{1,\dots,K-1\}$. Hence, $\lim_{n\to\infty}r_2(n)=0$. Note that $\tilde v_n^0=0$ and $\tilde v_n^{i+1}-\tilde v_n^i=\gamma$ for $i\geq h_n-k_n^1-1$. Thus $\tilde v_n$ is an admissible test function in the definition of $\widetilde B(\gamma)$, see \eqref{Bgamma2}, and we obtain \eqref{inf:int:1}. The proof of \eqref{inf:int:2} is similar.  A combination of \eqref{rev:inf:jump}--\eqref{inf:int:2} yields \eqref{est:lbjump} and finishes the proof of the liminf inequality.

\bigskip

\textbf{Limsup inequality.} To complete the $\Gamma$-convergence proof it is left to show that for every piecewise $H^1(0,1)$-function $v$ there exists a sequence $(v_n)$ converging to $v$ in $L^1(0,1)$ such that $\lim_nE_n^\ell(v_n)=E^\ell(v)$. We provide a recovery sequence only for functions $v$ which have a single jump, are constant close to the jump, and are sufficiently smooth away from the jump. It is straightforward to extend the construction to functions with finitely many jumps, and the claim follows by standard density and relaxation arguments.\\
Let $v$ be such that $S_v=\{t\}$ for some $t\in(0,1)$, $v\in C^2([0,1]\setminus \{t\})$ and $v\equiv v(t^-)$ on $(t-\rho,t)$ and $v\equiv v(t^+)$ on $(t,t+\rho)$ for some $\rho>0$ with $\rho<\min\{t,1-t\}$. Hence, there exists $\tilde v\in C^2(0,1)$ such that $v=\tilde v+ \chi_{(t,1]}(v(t^+)-v(t^-))$. Since $E^\ell(v)=+\infty$ if $v(t^+)<v(t^-)$, we can assume $v(t^+)>v(t^-)$. As in \cite[eq.~(130)]{BLO} and \cite[p.~680]{SSZ12}, a recovery sequence for $v$ is given by its piecewise affine interpolation with a small pertubation close to the jump which account for the boundary layer energy.

Fix $\eta>0$. The definition of $\widetilde B(\gamma)$ ensures the existence of a function $w:\N_0\to\R$ and of an $N\in\N$ such that $w^0=0$, $w^{i+1}-w^i=\gamma$ if $i\geq N$ and
\begin{align*}
& \sum_{j=2}^Kc_j\sum_{s=1}^{j-1}\frac{j-s}{j}J_1\left(w^{s}-w^{s-1}\right)+\sum_{j=2}^K\sum_{i\geq0}\bigg\{J_j\left(\frac{w^{i+j}-w^i}{j}\right)\\
&\qquad+\frac{c_j}{j} \sum_{s=i}^{i+j-1}J_1\left(w^{s+1}-w^s\right)-J_{0,j}(\gamma)\bigg\}\leq \widetilde B(\gamma)+\eta.
\end{align*}
Since the term in the infinite sum vanishes identically for $i\geq N$ we can replace the sum by any sum from $i=0$ to $\tilde N$ with $\tilde N \geq N$ without changing the estimate. Let $(k_n^1),(k_n^2),(h_n)\subset\N$ be such that $t-\frac{\rho}{2}\in[k_n^1,k_n^1+1)\lambda_n$, $t\in[h_n,h_n+1)\lambda_n$ and $t+\frac{\rho}{2}\in [k_n^2,k_n^2+1)\lambda_n$. Furthermore, let $n$ be large enough such that
\begin{equation}\label{assN}
N\leq \min\{h_n-k_n^1-K,~k_n^2-h_n-K\} 
\end{equation}
is satisfied. We define a sequence $(v_n)$ such that $v_n\in\A_n^{\#,\ell}(0,1)$ with help of $w$ and $\tilde v$ by  
\begin{equation*}
 v_n^i=\begin{cases}
        \tilde v(i\lambda_n)&\mbox{if $0\leq i\leq k_n^1$}\\
	v(t^-)-\sqrt{\lambda_n}(w^{h_n-i}-w^N+\gamma(i-h_n+N))&\mbox{if $k_n^1\leq i\leq h_n$,}\\
	v(t^+)+\sqrt{\lambda_n}(w^{i-(h_n+1)}-w^N-\gamma(i-h_n-1-N))&\mbox{if $h_n+1\leq i\leq k_n^2$,}\\
	v(t^+)+\tilde v(i\lambda_n)&\mbox{if $k_n^2\leq i\leq n$.}
       \end{cases}
\end{equation*}
By the definition of $w$, we have $w^{i+1}-w^i=\gamma$ if $i\geq N$, which implies that the terms with prefactor $\sqrt{\lambda_n}$ vanish for $i\leq h_n-N-1$ respectively $i\geq h_n+N+2$. It is not hard to check that $v_n\to v$ in $L^1(0,1)$ (see \cite[p.~680]{SSZ12} for related calculations). Next we show $\lim_{n\to\infty} E_n^\ell(v_n)=E^\ell(v)$. The definition of $v_n$ implies that
$$\frac{v_n^{h_n+j-s}-v_n^{h_n-s}}{\sqrt{\lambda_n}}=\frac{v(t+)-v(t-)}{\sqrt{\lambda_n}}+\mathcal O(1)\to+\infty\qquad\mbox{as $n\to\infty$,}$$ 
for all $j\in\{1,\dots,K\}$ and $s\in\{0,\dots,j-1\}$. Similar to \eqref{rev:inf:jump}, we obtain
\begin{align*}
 \sum_{j=2}^K\sum_{i=k_n^1}^{k_n^2}\zeta_{j,n}^i=&\sum_{j=2}^K\bigg\{\sum_{i=k_n^1}^{h_n-j}\zeta_{j,n}^i+c_j\sum_{s=1}^{j-1}\frac{j-s}{j}J_1\left(\gamma+\frac{v_n^{h_n-s+1}-v_n^{h_n-s}}{\sqrt{\lambda_n}}\right)+\sum_{i=h_n+1}^{k_n^2}\zeta_{j,n}^i\\
&\quad+c_j\sum_{s=1}^{j-1}\frac{j-s}{j}J_1\left(\gamma+\frac{v_n^{h_n+s+1}-v_n^{h_n+s}}{\sqrt{\lambda_n}}\right)\bigg\}-\sum_{j=2}^KjJ_{0,j}(\gamma)+r(n),
\end{align*}
where 
$$r(n):=\sum_{j=1}^K\sum_{s=-j+1}^0J_j\left(\gamma+\frac{v_n^{h_n+j+s}-v_n^{h_n+s}}{j\sqrt{\lambda_n}}\right)\to0\qquad\mbox{as $n\to\infty$.}$$
By the definition of $v_n$ and $w$ it follows for $n$ sufficiently large such that \eqref{assN} holds that
\begin{align}\label{sup:jump:a}
& \sum_{j=2}^K\bigg\{\sum_{i=k_n^1}^{h_n-j}\zeta_{j,n}^i+c_j\sum_{s=1}^{j-1}\frac{j-s}{j}J_1\left(\gamma+\frac{v_n^{h_n-s+1}-v_n^{h_n-s}}{\sqrt{\lambda_n}}\right)\bigg\}\notag\\
&\quad =\sum_{j=2}^K\sum_{i=0}^{h_n-k_n^1-j}\bigg\{J_j\left(\frac{w^{i+j}-w^i}{j}\right)+\frac{c_j}{j}\sum_{s=i}^{i+j-1}J_1(w^{s+1}-w^s)- J_{0,j}(\gamma)\bigg\}\notag\\
&\qquad+\sum_{j=2}^Kc_j\sum_{s=1}^{j-1}\frac{j-s}{j}J_1(w^{s+1}-w^s)\leq \widetilde B(\gamma)+\eta,
\end{align}
where we used $h_n-k_n^1-K\geq N$. In the same way, we obtain
\begin{equation}\label{sup:jump:b}
\sum_{j=2}^K\left\{\sum_{i=h_n+1}^{k_n^2}\zeta_{j,n}^i+c_j\sum_{s=1}^{j-1}\frac{j-s}{j}J_1\left(\gamma+\frac{v_n^{h_n+s+1}-v_n^{h_n+s}}{\sqrt{\lambda_n}}\right)\right\}\leq \widetilde B(\gamma)+\eta. 
\end{equation}
Let us now recover the integral term. A Taylor expansion of $J_j$ at $\gamma$ yields:
$$J_j(\gamma+z)=J_j(\gamma)+J_j'(\gamma)z+\frac12 J_j''(\gamma)z^2+\eta_j(z),$$
where $\frac{\eta_j(z)}{|z|^2}\to0$ as $z\to0$ for $j\in\{1,\dots,K\}$. Hence, using the definition of $\psi_j(z)=J_j(z)+c_jJ_1(z)$, $\psi_j'(\gamma)=0$ and $\alpha_j=\frac12\psi_j''(\gamma)$, we obtain for $z=\frac{1}{j}\sum_{s=1}^jz_s$ and $\omega(z):=\sup_{|t|\leq z}|\eta_j(t)|+j\sup_{|t|\leq z}|\eta_1(t)|$
\begin{align}\label{est:supexpansion}
J_j(\gamma+z)+\frac{c_j}{j}\sum_{s=1}^jJ_1(\gamma+z_s)-J_{0,j}(\gamma)\leq&\frac12J_j''(\gamma)\left(\frac{1}{j}\sum_{s=1}^jz_s\right)^2+\frac{c_j}{2j}J_1''(\gamma)\sum_{s=1}^jz_s^2+\omega(\max_{1\leq s\leq j}|z_s|)\notag\\
=&\frac{\alpha_j}{j}\sum_{s=1}^jz_s^2-\frac{J_j''(\gamma)}{2j^2}\sum_{s=1}^{j-1}\sum_{m=s+1}^j(z_s-z_m)^2+\omega(\max_{1\leq s\leq j}|z_s|), 
\end{align}
where we use in the last step:
$$\left(\sum_{s=1}^ja_s\right)^2=\sum_{s=1}^ja_s^2+2\sum_{s=1}^{j-1}\sum_{m=s+1}^ja_sa_m=j\sum_{s=1}^ja_s^2-\sum_{s=1}^{j-1}\sum_{m=s+1}^j(a_s-a_m)^2.$$
Combining \eqref{est:supexpansion} with $v_n^i=v(i\lambda_n)$ for all $i\in\{0,\dots,h_n-N-1\}\cup\{h_n+N+2,\dots,n\}$ and the $C^2$-regularity of $v$ away from the jump, we find $r_2(n)$ satisfying $r_2(n)\to0$ as $n\to\infty$ such that:
\begin{align*}
 \zeta_{j,n}^i\leq\lambda_n\left\{\frac{\alpha_j}{j}\sum_{s=i}^{i+j-1}\left(\frac{ v_n^{s+1}- v_n^s}{\lambda_n}\right)^2+r_2(n)\right\},
\end{align*}
for $i\in\{0,\dots,k_n^1-1\}\cup\{k_n^2+1,\dots,n-1\}=:Q_n$. Hence, we have for $n$ large enough
\begin{align}\label{sup:elastic}
\sum_{j=2}^K\sum_{i\in Q_n}\zeta_{j,n}^i&\leq\sum_{j=2}^K\frac{\alpha_j}{j}\lambda_n\sum_{i\in Q_n}\sum_{s=i}^{i+j-1}\left(\frac{ v_n^{s+1}- v_n^s}{\lambda_n}\right)^2+r_2(n)\notag\\
&=\sum_{j=2}^K\alpha_j\lambda_n\sum_{i\in Q_n}\left(\frac{v_n^{i+1}-v_n^i}{\lambda_n}\right)^2+r_2(n)\notag\\
&=\alpha \int_{(0,1)\setminus (t-\frac{\rho}{4},t+\frac{\rho}{4})} v_n'(x)^2\,dx+r_2(n),
\end{align}
%
where we use in the second step the periodicity of $v_n^i$ and $v_n^{i+1}-v_n^i=0$ for $i\in\{k_n^1,\dots,k_n^1+K\}\cup\{k_n^2-K,\dots,k_n^2\}$ for $n$ sufficiently large. Since $v_n$ is the piecewise affine interpolation of $\tilde v$ on $(0,t-\frac{\rho}{4})$ and $v(t+)+\tilde v$ on $(t+\frac{\rho}{4},1)$ and $\tilde v'=0$ on $(t-\rho,t+\rho)$, a combination of \eqref{sup:jump:a}, \eqref{sup:jump:b}, and \eqref{sup:elastic} yields
\begin{align*}
 \limsup_{n\to\infty}E_n^\ell(v_n)=\limsup_{n\to\infty}\sum_{j=2}^K\sum_{i=0}^{n-1}\zeta_{j,n}^i\leq \alpha \int_0^1\tilde v'(x)^2\,dx+2\widetilde B(\gamma)-\sum_{j=2}^Kj\psi_j(\gamma)+2\eta
\end{align*}
and the claim follows from $\|\tilde v'\|_{L^2((0,1))}=\|v'\|_{L^2((0,1))}$ and by the arbitrariness of $\eta>0$, where we note that $v'$ denotes the absolutely continuous part of the derivative of $v$ only. 

\bigskip

\textbf{Convergence of minimization problems.} The convergence of minimal energies follows from the coerciveness of $E_n$ and the $\Gamma$-convergence result. Regarding the coerciveness, we recall that $\sup_n E_n^\ell(v_n)<\infty$ yields $\sup_n\|v_n'\|_{L^1(0,1)}<\infty$ and thus there exists a sequence of constants $c_n$ such that $v_n-c_n$ is equibounded in $L^1(0,1)$  and by the discussion below \eqref{ene:coer} we obtain compactness of $v_n-c_n$ in $L^1(0,1)$. Moreover $\min_v E^\ell(v)=\min\{\alpha \ell^2,\beta\}$.
\end{proof}

\begin{proof}[Proof of Lemma~\ref{prop:beta}]
We prove that 
\begin{equation}\label{btb}
 B(\gamma)-\frac12J_1(\gamma)=\widetilde B(\gamma)-\frac12\sum_{j=2}^Kjc_jJ_1(\gamma),
\end{equation}
where $B(\gamma)$ and $\widetilde B(\gamma)$ are given in \eqref{Bgamma} and \eqref{Bgamma2}. This equality implies the assertion since 
$$\beta=2B(\gamma)-\sum_{j=1}^KjJ_j(\gamma)=2\widetilde B(\gamma)-\sum_{j=2}^Kj(J_j(\gamma)+c_jJ_1(\gamma))=2\widetilde B(\gamma)-\sum_{j=2}^Kj\psi_j(\gamma).$$
Let $u:\N_0\to\R$ be a candidate for the minimum problems defining $B(\gamma)$ and $\widetilde B(\gamma)$, i.e.~$u^0=0$ and $u^{i+1}-u^i=\gamma$ for $i\geq N$ for some $N\in\N_0$. Then it holds for the infinite sum in the definition of $B(\gamma)$ that
\begin{align*}
 &\sum_{j=2}^K\sum_{i\geq0}\bigg\{J_j\left(\frac{u^{i+j}-u^i}{j}\right)+\frac{c_j}{j} \sum_{s=i}^{i+j-1}J_1\left(u^{s+1}-u^s\right)-J_{0,j}(\gamma)\bigg\}\\
&=\sum_{i=0}^{N-1}\bigg\{\sum_{j=2}^KJ_j\left(\frac{u^{i+j}-u^i}{j}\right)-\sum_{j=2}^KJ_{0,j}(\gamma)\bigg\}+\sum_{j=2}^K\frac{c_j}{j}\sum_{i=0}^{N-1}\sum_{s=i}^{i+j-1}J_1\left(u^{s+1}-u^s\right).
\end{align*}
For given $j\in\{2,\dots,K\}$, the nearest neighbor terms on the right-hand side above can be rewritten as
\begin{align*}
 \frac1j\sum_{i=0}^{N-1}\sum_{s=i}^{i+j-1}J_1\left(u^{s+1}-u^s\right)
=&\sum_{i=0}^{N-1}J_1\left(u^{i+1}-u^i\right)-\sum_{s=1}^{j-1}\frac{j-s}{j}J_1\left(u^{s}-u^{s-1}\right)\\
&+\frac{1}{j}\sum_{i=N-j+1}^{N-1}\sum_{s=N}^{i+j-1}J_1\left(u^{s+1}-u^s\right),
\end{align*}
%
%
where the third term on the right-hand side above simplifies since $u^{i+1}-u^i=\gamma$ for $i\geq N$ to
\begin{align*}
 \frac{1}{j}\sum_{i=N-j+1}^{N-1}\sum_{s=N}^{i+j-1}J_1\left(u^{s+1}-u^s\right)=\frac12(j-1)J_1(\gamma).
\end{align*}
%
%
Combining the previous identities and recalling once more $\sum_{j=2}^Kc_j=1$ and $\sum_{j=2}^KJ_{0,j}(\gamma)=\sum_{j=1}^KJ_j(\gamma)=J_{CB}(\gamma)$, we obtain that
\begin{align*}
& \sum_{j=2}^Kc_j\sum_{s=1}^{j-1}\frac{j-s}{j}J_1\left(u^{s}-u^{s-1}\right)+\sum_{j=2}^K\sum_{i\geq0}\bigg\{J_j\left(\frac{u^{i+j}-u^i}{j}\right)\\
&+\frac{c_j}{j} \sum_{s=i}^{i+j-1}J_1\left(u^{s+1}-u^s\right)-J_{0,j}(\gamma)\bigg\}-\frac12\sum_{j=2}^Kjc_jJ_1(\gamma)\\
&=\sum_{i=0}^{N-1}\bigg\{\sum_{j=2}^KJ_j\left(\frac{u^{i+j}-u^i}{j}\right)-\sum_{j=2}^KJ_{0,j}(\gamma)\bigg\}+\sum_{j=2}^Kc_j\sum_{i=0}^{N-1}J_1\left(u^{i+1}-u^i\right)-\frac12J_1(\gamma)\\
&=\sum_{i=0}^{N-1}\bigg\{\sum_{j=1}^KJ_j\left(\frac{u^{i+j}-u^i}{j}\right)-J_{CB}(\gamma)\bigg\}-\frac12J_1(\gamma)\\
&=\sum_{i\geq0}\bigg\{\sum_{j=1}^KJ_j\left(\frac{u^{i+j}-u^i}{j}\right)-J_{CB}(\gamma)\bigg\}-\frac12J_1(\gamma).
\end{align*}
By the arbitrariness of $u:\N_0\to\R$ and $N\in\N_0$ with $u^0=0$ and $u^{i+1}-u^i=\gamma$ for $i\geq N$ and the definition of $B(\gamma)$ and $\widetilde B(\gamma)$, see \eqref{Bgamma} and \eqref{Bgamma2}, the equality \eqref{btb} and thus the lemma is proven.
\end{proof}

\begin{remark}
 In the case of nearest and next-to-nearest neighbour interactions, i.e.~$K=2$, we have $c_2=c_K=1$. Thus it holds $\psi_j(z)=J_2(z)+J_1(z)=J_{CB}(z)$ and the boundary layer energy $\widetilde B(\gamma)$ is given by
\begin{equation*}
 \begin{split}
 \widetilde B(\gamma)=&\inf_{k\in\mathbb{N}_0}\min\bigg\{\frac{1}{2}J_1\left(u^{1}-u^{0}\right)+\sum_{i\geq0}\bigg\{J_2\left(\frac{u^{i+2}-u^i}{2}\right)+\frac12\sum_{s=i}^{i+1}J_1(u^{s+1}-u^s)-J_{CB}(\gamma)\bigg\}:\\
&\hspace*{2cm} u:\mathbb N_0\to \mathbb R,u^{0}=0,u^{i+1}-u^{i}=\gamma\mbox{ if $i\geq k$}\bigg\}.
 \end{split}
\end{equation*}
This coincides with the definition of the (free) boundary layer energy $B$ and $B(\gamma)$  defined in \cite{BLO} and \cite{BC07,SSZ11} respectively. The jump energy $\beta$ then reads 
$$\beta=2\widetilde B(\gamma)-2J_{CB}(\gamma),$$
and coincides with the corresponding jump energies defined in \cite{BC07,BLO,SSZ11}.
\end{remark}

\begin{remark}\label{blofail}
The proof of Theorem~\ref{Th1} is very similar to the proof of \cite[Theorem 4]{BLO} for a related statement for multibody potentials with finite range interactions. Next, we briefly discuss why \cite[Theorem 4]{BLO} is not directly applicable for Lennard-Jones systems, see \eqref{def:lj}, with $K>2$. In \cite{BLO}, a lower-bound comparison potential is introduced, which in our notation reads
$$\Phi_-(z):=\inf\left\{\sum_{j=1}^K\sum_{i=1}^{K-j+1}\frac{1}{K-j+1}J\left(\sum_{s=i}^{i+j-1}z_s\right):\sum_{s=1}^Kz_s=Kz\right\},$$
cf.~\cite[eq.~(8)]{BLO}. It is assumed that $\Phi_-$ has a unique minimizer $z_{\min}$ and the infimum in the definition of $\Phi_-(z_{\min})$ is attained for $z_s=z_{\min}$ for $s=1,\dots,K$, cf.~\cite[assumptions (v), (vi) on p.157, see also Remark~2]{BLO}. This is in general not satisfied by Lennard-Jones potentials \eqref{def:lj} for $K>2$. Consider $K=3$. In this case the term in the infimum problem in the definition of $\Phi_-(z_{\min})$ reads
\begin{align*}
 \frac{1}{3}\left\{J(z_1)+J(z_2)+J(z_3)\right\}+\frac{1}{2}\left\{J\left(z_1+z_2\right)+J\left(z_1+z_2\right)\right\}+J\left(3z_{\min}\right),
\end{align*}
where $z_1+z_2+z_3=3z_{\min}$. Assume by contradiction that the infimum is attained for $z_1=z_2=z_3=z_{\min}$. The optimality conditions yield that there exists $\lambda\in\R$ such that $\tfrac13J'(z_{\min})+\frac12J'(2z_{\min})=\lambda$ (condition for $z_1=z_3=z_{\min}$) and $\frac13J'(z_{\min})+J'(2z_{\min})=\lambda$ (condition for $z_2=z_{\min}$). Hence, $J'(2z_{\min})=0$ and thus $z_{\min}=\frac12\delta_1$, where $\delta_1$ denotes the unique minimizer of $J$ as is given in \eqref{LJ:deltaj}. The unique minimizer $\gamma$ of $J_{CB}$ satisfies $\gamma>\frac12 \delta_1$, see \eqref{gammaK}. By the definition of $\Phi_-$, it holds $\Phi_-(z)\leq J_{CB}(z)$, and by assumption we have $\inf_{z\in \R}\Phi_-(z)=\Phi_-(\delta_2)=J(\delta_2)+J(2\delta_2)+J(3\delta_2)= J_{CB}(\delta_2)$. Hence, 
$$\Phi_-(\gamma)\leq J_{CB}(\gamma)<J_{CB}(\delta_2)=\Phi_-(\delta_2)=\inf_{z\in \R}\Phi_-(z)\leq \Phi_-(\gamma),$$
which is a contradiction. Hence, the Lennard-Jones potentials \eqref{def:lj} do not satisfy the assumptions on $\Phi_-$ in the case $K=3$. This argument can be adapted for all $K>2$.
\end{remark}

\textbf{Acknowledgment.} This work was partly supported by a grant of the Deutsche Forschungsgemeinschaft (DFG) SCHL 1706/2-1.


\begin{thebibliography}{99}
 \bibitem{AFP} L.~Ambrosio, N.~Fusco and D.~Pallara, \textit{Functions of bounded variation and free discontinuity problems}, Oxford Univ.~Press, 2000.
\bibitem{Ball} J.M. Ball, Some open problems in elasticity. In \textit{Geometry, mechanics and dynamics}, Springer, New York, 2002, pp.\ 3--59.
\bibitem{BLBLi} X.~Blanc, C.~LeBris and P.~L.~Lions, From molecular models to continuum mechanics, \textit{Arch.~Ration.~Mech.~Anal.} \textbf{164} (2002), no.~4, 341–381. 
\bibitem{B02} A.~Braides, $\Gamma$-\textit{Convergence for Beginners}, Oxford Univ.~Press, 2002.
\bibitem{BC07} A.~Braides and M.~Cicalese, Surface energies in nonconvex discrete systems, \textit{Math.~Models Methods Appl.~Sci.} \textbf{17} (2007), 985--1037.
 \bibitem{BDG99} A.~Braides, G.~Dal Maso and A.~Garroni, Variational formulation of softening phenomena in fracture mechanics: The one-dimensional case, \textit{ Arch.~Rational Mech.~Anal.} \textbf{146} (1999), 23--58.~
\bibitem{BG02} A.~Braides and M.~S.~Gelli, Continuum Limits of Discrete Systems without Convexity Hypotheses, \textit{Math.~Mech.~Solids} \textbf{7} (2002), no.~1, 41--66. 
\bibitem{BG04} A.~Braides and M.~S.~Gelli, The passage from discrete to continuous variational problems: a nonlinear homogenization process, in \textit{Nonlinear Homogenization and its Applications to Composites, Polycrystals and Smart Materials}  NATO Sci. Ser. II Math. Phys. Chem., 170, Kluwer Acad. Publ., Dordrecht, 2004.
\bibitem{BG07} A.~Braides and M.~S.~Gelli, From discrete systems to continuous variational problems: An introduction, \textit{Lect.\ Notes Unione Mat.\ Ital.} \textbf{2} (2006) 3--77.
\bibitem{BG15} A.~Braides and M.~S.~Gelli, Asymptotic analysis of microscopic impenetrability constraints for atomistic systems, \textit{J.\ Mech.\ Phys.\ Solids} \textbf{96} (2016), 235--251.
\bibitem{BGS} A.~Braides, M.~S.~Gelli and M.~Sigalotti, The passage from nonconvex discrete systems to variational problems in Sobolev spaces: the one-dimensional case, \textit{Proc. Steklov Inst. Math.} \textbf{236} (2002), no.~1, 395–414.  
\bibitem{BLO} A.~Braides, A.~Lew and M.~Ortiz, Effective cohesive behavior of layers of interatomic planes, \textit{Arch.~Rational Mech.~Anal.} \textbf{180} (2006), 151--182.
\bibitem{BS14} A.~Braides and M.~Solci, Asymptotic analysis of Lennard-Jones systems beyond the nearest-neighbour setting: A one-dimensional prototypical case, \textit{Math.~Mech.~Solids}, \textbf{21} (2016), 915--930.
\bibitem{BT} A.~Braides and L.~Truskinovsky, Asymptotic expansions by $\Gamma$-convergence, \textit{Cont. Mech. Thermodyn.} \textbf{20} (2008), 21--62.
\bibitem{FS14} M.~Friedrich and B.~Schmidt, An atomistic-to-continuum analysis of crystal cleavage in a two-dimensional model problem, \textit{J. Nonlinear Sci.} \textbf{24} (2014), no. 1, 145–183. 
\bibitem{FS15} M.~Friedrich and B.~Schmidt, An analysis of crystal cleavage in the passage from atomistic models to continuum theory, \textit{Arch.\ Ration.\ Mech.\ Anal.} \textbf{217} (2015), no. 1, 263--308. 
\bibitem{GPPS} M.~G.~D.~Geers, R.~H.~J.~Peerlings, M.~A.~Peletier and L.~Scardia, Asymptotic behaviour of a pile-up of infinite walls of edge dislocations, \textit{Arch.~Ration.~Mech.~Anal.} \textbf{209} (2013), no. 2, 495--539. 
\bibitem{LO} M. Luskin and C. Ortner, Atomistic-to-continuum coupling, \textit{Acta Numer.} \textbf{22} (2013), 397--508. 
\bibitem{SSZ11} L.~Scardia, A.~Schl\"omerkemper and C.~Zanini, Boundary layer energies for nonconvex discrete systems, \textit{Math.~Models Methods Appl.~Sci.} \textbf{21} (2011), 777--817.
\bibitem{SSZ12} L.~Scardia, A.~Schl\"omerkemper and C.~Zanini, Towards uniformly $\Gamma$-equivalent theories for nonconvex discrete systems, \textit{Discrete Contin.~Dyn.~Syst.~Ser.~B} \textbf{17} (2012), no.~2, 661--686.
\bibitem{SS14} M.~Sch\"affner and A.~Schl\"omerkemper, On a $\Gamma$-convergence analysis of a quasicontinuum method, \textit{Multiscale Model.~Simul.} \textbf{13} (2015), 132--172.
\bibitem{S15} M.~Sch\"affner, \textit{Multiscale analysis of non-convex discrete systems via $\Gamma$-convergence}, Dissertation, University of Würzburg, 2015, {\tt http://nbn-resolving.de/urn:nbn:de:bvb:20-opus-122349}.
\bibitem{TOP} E. Tadmor, M. Ortiz and R. Phillips, Quasicontinuum analysis of defects in solids, \textit{Phil. Mag.} \textbf{A} \textbf{73} (1996), 1529--1563.
\bibitem{T96} L.\ Truskinovsky, Fracture as a phase transition, in \textit{Contemporary Research in the Mechanics and Mathematics of Marterials} (1996), 322--332.
\end{thebibliography}
\end{document}